\documentclass[ASNA,onecolumn]{USG}
\usepackage{anyfontsize}
\usepackage{colortbl}
\usepackage{xcolor}
\usepackage{steinmetz}
\usepackage{float}
\usepackage{amsmath,amsfonts}
\usepackage{amssymb}
\usepackage{algorithm}
\usepackage{algpseudocode}
\usepackage{array}
\usepackage{caption}
\usepackage{booktabs}
\usepackage{enumitem}
\usepackage{graphicx}
\usepackage{amsthm}

\hypersetup{hypertexnames=false}

\graphicspath{{./images/}}
\specialissue{Morphologic Hard and Soft Tissue Changes following Tooth
Extraction in Different Maxilla and Mandible Anatomic Areas: Current
Knowledge, Prognostic Factors Analysis, and Clinical Management}

\articletype{RESEARCH ARTICLE}%
\subarticletype{Particle Technology and Fluidization}

\received{}
\revised{}
\accepted{}
\journal{Mathematical Methods in the Applied Sciences}
\volume{0}
\copyyear{2026}
\startpage{1}
\articledoi{}

\begin{document}
\title{Manifold Fractional Harmonic Transform for 3D Point Clouds}
\transtitle{Guidelines for Establishing a Cytometry Laboratory}
\subtranstitle{trans-subtitle}
\author[1,2]{Jiamian Li}[https://orcid.org/0009-0006-2633-4705]
\author[1,2]{Bing-Zhao Li}[https://orcid.org/0000-0002-3850-4656]





















\authormark{TAYLOR \textsc{et al.}}
\titlemark{PLEASE INSERT YOUR ARTICLE TITLE HERE}

\address[1]{\orgdiv{School of Mathematics and Statistics, }\orgname{Beijing Institute of Technology, }%
\orgaddress{\state{Beijing 100081, }\country{China}}}

\address[2]{\orgdiv{Beijing Key Laboratory on MCAACI,}\orgname{Beijing Institute of Technology, }%
\orgaddress{\state{Beijing 100081, }\country{China}}}



\corres{Bing-Zhao Li  (\email{li\_bingzhao@bit.edu.cn})}

\editor{\textbf{Academic Editor:}   ~|~ \textbf{Guest Editor:}  }

\presentaddress{This is sample for present address text this is sample for present address text.}

\fundingInfo{
Natural Science Foundation of Beijing Municipality, China, Grant/Award Number: 4242011.}


\transkeywords{DEM | flexible cylindrical particle | internal friction angle | particle deformation | shear stress}

\abstract[ABSTRACT]{Point clouds can be regarded as discrete samples of smooth manifolds and are typically analyzed via the eigenfunctions of the Laplace-Beltrami operator. This paper extends manifold spectral analysis to the fractional domain, enabling continuous interpolation between the spatial and spectral domains for point cloud data. First, a point cloud manifold fractional harmonic transform (PMFHT) is proposed, with its fundamental properties rigorously derived, along with the associated convolution, correlation, and sampling theorems. These theoretical results establish a solid foundation for stable fractional-order spectral representation on manifolds. Second, within the PMFHT framework, two representative algorithms are developed. On the one hand, by integrating multi-order PMFHT with chaotic phase modulation, a point cloud encryption scheme is constructed, characterized by a large key space and high sensitivity to key perturbations. On the other hand, an optimal filter is designed in the fractional manifold spectral domain, leading to a maritime target detection method specifically tailored for point cloud data, which effectively suppresses sea clutter while preserving weak target energy under low signal-to-clutter ratio conditions. Finally, experiments on measured data validate the effectiveness of the proposed algorithms.}

 \transabstract[transABSTRACT]{This is a generic template designed for use by multiple journals,
which includes several options for customization. Please
refer the author guidelines and author LaTeX manuscript preparation document for the journal to which you
are submitting in order to confirm that your manuscript will
comply with the journal’s requirements. Please replace this
text with your abstract. This is sample abstract text just for the template display purpose.}

\abbr{5-FU, 5-fluorouracil; CFD, computational fluid dynamics; CH, channel; EFS, event-free survival; GBM, glioblastoma multiforme; OS, overall survival; PFS, progression-free
survival; SD, standard deviation.}

\contributed{Hao Zhang and Pengyue D. Guo contributed equally to this study.}

\dedicated{Dedicated to Srivivasa Ramanujano on the occasion of his 125th birth anniversary.}

\copyright{
  }


\maketitle


\section{Introduction}\label{sec1}

Point cloud data is a collection of discrete points in three-dimensional (3D) space. Each point contains its spatial coordinates $(x, y, z)$ and may contain information such as color and reflectivity. Point clouds are typically generated via active or passive sensing modalities, such as LiDAR and photogrammetry \cite{han2017review},\cite{zhang2023deep}. 
As a fundamental 3D representation method, point clouds have been widely used in fields such as autonomous driving, object detection, and digital twins due to their simplicity, flexibility, and powerful expressiveness \cite{guo2020deep}, which has also promoted the rapid development of the research field of point cloud processing.
Manifolds are foundational constructs in differential geometry and topology, characterizing spaces that are locally Euclidean yet possess intricate global structures \cite{lee2003smooth}. In physical contexts, smooth and continuous object surfaces are frequently modeled as 2D manifolds. Consequently, point clouds can be viewed as discretized samplings of these underlying continuous manifolds \cite{MARTIN2011427}, \cite{lee2022statistical}, \cite{pauly2001spectral}. The framework of point cloud manifolds bridges the gap between unordered discrete points and continuous differential geometry, facilitating the integration of manifold-based analytical tools into point cloud processing.

Vallet et al. introduced a seminal mesh-based framework for spectral geometry processing using manifold harmonics, which constructs an orthonormal basis from the eigenfunctions of the Laplace–Beltrami operator (LBO) defined on triangulated surfaces \cite{vallet2008spectral}. Building upon discrete exterior calculus, their manifold harmonic transform (MHT) provides a Fourier-like spectral representation that enables geometry-aware filtering, compression, and various shape analysis tasks. This foundational work laid the theoretical groundwork for spectral analysis on geometric domains.
Liu et al. achieved a significant advance by extending spectral geometry processing from meshes to point-sampled manifolds and introducing a point-based manifold harmonic basis \cite{liu2012point}. Their method constructs a convergent and symmetrizable discrete LBO via local Voronoi cell estimation on tangent planes, ensuring orthogonality of the resulting basis functions without requiring global mesh connectivity. The resulting point-based manifold harmonic transform enables direct spectral analysis of point clouds, supporting tasks such as spectral filtering, feature extraction, and noise removal. 
This seminal contribution laid the foundation for harmonic analysis on point clouds, offering provably improved convergence behavior and greater geometric fidelity than graph-based or naively symmetrized Laplacian operators.
In this work, we adopt the terms point cloud manifold harmonic bases (PMHB) and point cloud manifold harmonic transform (PMHT) to emphasize the focus on point cloud data.

The fractional Fourier transform (FRFT) is a generalized form of the classical Fourier transform and can be interpreted as rotating a signal in the time–frequency plane by an arbitrary angle. The traditional Fourier transform maps a signal from the time domain to the frequency domain, equivalent to a 90° rotation, whereas the FRFT allows rotation by arbitrary angles, revealing intermediate states between the time and frequency domains (i.e., fractional domains). This enables simultaneous characterization of both time domain and frequency domain properties, providing a powerful and flexible tool for nonstationary signal analysis \cite{BULTHEEL2004182}, \cite{CHEN202171}, \cite{FU2023211}. 
Motivated by this, we extend PMHT to the fractional domain. Specifically, we propose the point cloud manifold fractional harmonic transform (PMFHT), which constructs a continuous fractional spectral domain between the spatial and frequency domains by nonlinearly scaling the eigenvalues of LBO, thereby enabling a more flexible spectral representation.

Point cloud data, due to its inclusion of 3D structural information, has been widely applied in surveying, remote sensing, and autonomous driving. However, research on information security and clutter suppression for 3D point clouds remains relatively limited.
In information security, FRFT-based cryptosystems offer a larger key space compared to traditional Fourier domain methods by using the transform order as the key. For example, Tao et al. used multi-order FRFT to encrypt images by summing interpolated sub-images in different fractional domains, significantly improving the system's resistance to brute-force attacks \cite{5551197}. Ben Farah et al. proposed a triple FRFT-based encryption algorithm that integrates encoding with exclusive-or operations and a Lorenz chaotic map to achieve high sensitivity and strong obfuscation \cite{FARAH2020105777}. With the increasing prevalence of 3D data in sensitive applications, Liu et al. proposed a novel privacy protection scheme for 3D point cloud classification based on optical chaos generated by mutually coupled lasers \cite{Liu:23}. This method employs both permutation and diffusion processes to ensure that the geometric features of the point cloud are effectively hidden. 
In the field of clutter suppression, taking sea clutter as an example, current research mainly includes:
Graph neural network-based methods represent radar data as graph-structured signals to utilize their inherent spatiotemporal correlations. Experiments show that such methods outperform traditional convolutional neural network models, demonstrating that graph representations can better utilize contextual spatiotemporal information to handle the complexity of sea clutter and the similarity between targets and clutter \cite{ru2023marine}, \cite{su2024radar}.
Another research direction adopts an information geometry perspective, enhancing feature robustness by revealing the inherent geometric structure in the data. Radar observation data is modeled as points on a manifold, and a constant false alarm rate detector is designed on the resulting Riemannian manifold. Then, targets and clutter are distinguished by calculating geometric distances on the manifold \cite{hua2021target}, \cite{cao2021nonstationary}, \cite{li2025novel}. 
In FRFT-based weak target detection in the ocean, the dispersed target energy is concentrated into a peak on the time-frequency plane by selecting the optimal FRFT, thus achieving target detection under low signal-to-clutter ratio (SCR) conditions. To suppress tsunami pulses in sea clutter, these methods typically integrate cancellation, filtering, or data selection mechanisms to remove clutter components inconsistent with target characteristics \cite{chen2013detection}, \cite{gao2021weak}.

However, these clutter suppression methods are difficult to directly apply to 3D point cloud data.
Building upon the proposed PMFHT theoretical framework, this paper leverages the multi-order and chaotic characteristics of the forward and inverse transforms to develop a novel encryption algorithm, significantly enlarging the key space and enhancing information security. 
Furthermore, based on the minimum mean square error (MSE) criterion, this paper designs an optimal fractional domain filter suitable for 3D point clouds. By determining the optimal fractional transform order, the proposed method outperforms traditional methods, effectively suppressing heterogeneous sea clutter and significantly improving the SCR.

This article is organized as follows. Section~\ref{sec:2} reviews the fundamentals of point cloud manifold harmonic analysis and the FRFT. Section~\ref{sec:3} introduces PMFHT, including its definition,  properties, and theorems. Section~\ref{sec:4} constructs the 3D point cloud encryption algorithm based on PMFHT and the optimal fractional domain filtering framework under the minimum MSE criterion. Section~\ref{sec:6} presents the application experiments, where the effectiveness of the two proposed algorithms is validated using available real-world datasets.
Finally, Section~\ref{sec:7} concludes the paper.
 
 \section{Preliminaries}\label{sec:2}

\subsection{Point Cloud Manifold Harmonic Analysis}
PMHT extends the classical manifold harmonic framework to point clouds, enabling spectral analysis without requiring explicit mesh connectivity~\cite{liu2012point}. 
The resulting bases are orthogonal, which is ensured by constructing a symmetrizable discrete LBO directly on the point clouds.
\begin{definition}
For a compact Riemannian manifold ${M}$ without boundary, the LBO is defined as:
\begin{equation}
\Delta_{{M}} f = \mathrm{div} \, \mathrm{grad} \, f.
\end{equation}

The eigenvalue problem of the LBO is given by:
\begin{equation}
\Delta_{{M}} H = -\lambda H,
\end{equation}
where $\lambda$ and $H$ are the eigenvalue and eigenfunction, respectively, and the negative sign is introduced to ensure that all eigenvalues $\lambda$ are nonnegative. 
The eigenfunctions form an orthonormal basis with respect to the $L^2(M)$ inner product:
\begin{equation}
\langle H_i, H_j \rangle = \delta_{ij},
\end{equation}
where $H_i$ and $H_j$ are the eigenfunctions corresponding to the $i^{\text{th}}$ and $j^{\text{th}}$ eigenvalues, respectively, and $\delta_{ij}$ is the Kronecker delta function, satisfying $\delta_{ij} = 1$ if $i = j$ and $\delta_{ij} = 0$ otherwise.

Given a point cloud $P = \{p_1,\dots,p_N\}$ sampled from ${M}$, the construction of the PMHB begins with the pointwise approximation of the LBO $\Delta_{{M}} f(p)$. The continuous integral representation of $\Delta_{{M}}$ can be written as \cite{BELKIN20081289}:

\begin{equation}
		\Delta_{{M}} f(p)= 
		\lim_{t \to 0} \frac{1}{4 \pi t^2} \Bigg[
		\int_{{M}} e^{-\frac{\|p-y\|^2}{4t}} f(p) d\mu(y) - \int_{{M}} e^{-\frac{\|p-y\|^2}{4t}} f(y) d\mu(y)
		\Bigg].
\end{equation}
 
%
For each point $p \in P$, estimate its local tangent plane $\hat{T}_p$
by finding the best-fitting plane of its $r$-neighborhood
$P_r = P \cap B(p, r)$, where $B(p,r)$ denotes the ball centered at
$p$ with radius $r = 10\varepsilon$ ($\varepsilon > 0$ ).
The choice $r = 10\varepsilon$ ensures that $\hat{T}_p$ converges to the true tangent plane as the sampling
becomes denser, as proved in Belkin et al.\ \cite{Belkin2009}.
Next, project the neighbors $P_\delta = P \cap B(p, \delta)$, with
$\delta \ge 10\varepsilon$, onto $\hat{T}_p$ and compute their Euclidean
Voronoi diagram on $\hat{T}_p$. The area of the Voronoi cell
$\mathrm{vol}\!\left(\mathrm{Vor}_{\hat{T}_p}(p)\right)$ serves as an
approximation of the local area element around $p$ on ${M}$.

The discrete approximation of $\Delta_{{M}} f(p)$ is then:
\begin{equation}
\hat{\Delta}^{\,t}_{P} =
\frac{1}{4 \pi t^2} \sum_{q \in P_\delta}
e^{-\frac{\|q-p\|^2}{4t}}
\bigl(f(q) - f(p)\bigr) \, \mathrm{vol}\bigl(\mathrm{Vor}_{\hat{T}_q}(q)\bigr),
\label{eq:discrete-lbo}
\end{equation}
where $\mathrm{vol}\!\left(\mathrm{Vor}_{\hat{T}_q}(q)\right)$ denotes the
area of the Voronoi cell of $q$ computed on its own estimated tangent
plane $\hat{T}_q$.
As the sampling becomes denser and $t(\varepsilon) = \varepsilon^{1/2+a}$ with $a>0$, $\hat{\Delta}^{\,t}_{P} f(p)$ converges pointwise to 
$\Delta_{{M}} f(p)$. 

Collecting \eqref{eq:discrete-lbo} for all $p_i \in P$ yields the matrix form:
\begin{equation}
\hat{\Delta}^{\,t}_{P} f =\hat{L}^{\,t}_{P} f,
\label{6}
\end{equation}
where $\hat{L}^{\,t}_{P} = B^{-1} {Q}$ is symmetrizable with:
\begin{equation}
\begin{aligned}
q_{ij} &= 
\frac{\mathrm{vol}(\mathrm{Vor}_{\hat{T}_{p_i}})\mathrm{vol}(\mathrm{Vor}_{\hat{T}_{p_j}})}{4 \pi t^2}
\exp\left(-\frac{\|p_i - p_j\|^2}{4t}\right), i \neq j \\
q_{ii} &= - \sum_{j \neq i} q_{ij},\quad
b_{ii} = \mathrm{vol}(\mathrm{Vor}_{\hat{T}_{p_i}}).
\end{aligned}
\end{equation}
\end{definition}
Solving the generalized eigenproblem:
\begin{equation}
Q H = -\lambda B H,
\end{equation}
yields the eigenvalues $\{\lambda_i\}$ and the corresponding eigenvectors $\{H_i\}$. 
The eigenvectors, namely the PMHB, serve as a set of basis functions that enable a Fourier-like spectral decomposition of any function $f$ sampled on $P$:

\begin{equation}
\begin{aligned}
\tilde{f}_i &= \langle f, H_i \rangle = f^\top {B} H_i, \\
f &= \sum_i \tilde{f}_i H_i.
\end{aligned}
\end{equation}


\subsection{The Fractional Fourier Transform}
\begin{definition}
The FRFT of a signal \(x(t)\) is defined as \cite{Kamalakkannan}:
\begin{equation}
X_p(u) = \mathcal{F}_p[x](u) = \int_{-\infty}^{+\infty} x(t) \, K_p(t,u) \, dt,
\tag{10}
\label{10}
\end{equation}
where
$
K_p(t,u) = \sqrt{1 - j \cot \beta}\;
\exp\left\{ j\pi \left( t^2 \cot \beta - 2tu \csc \beta + u^2 \cot \beta \right) \right\},
$
with \(\beta = p\pi/2\) and \(\beta \neq n\pi\).  
For special cases of \(\beta\), the kernel becomes
$
K_p(t,u) =
\begin{cases}
\delta(t-u), & \beta = 2n\pi,\\[4pt]
\delta(t+u), & \beta = (2n \pm 1)\pi.
\end{cases}
$
\end{definition}
\begin{remark}
Here, \(\beta\) represents the rotation angle of the FRFT in the time–frequency plane, \(p\) denotes the transform order, and \(\mathcal{F}_p(\cdot)\) is the FRFT operator. It is evident that the FRFT is periodic with period 4. In particular, for \(p = 4n+1\) (i.e., \(\beta = 2n\pi + \pi/2\)), the FRFT reduces to the conventional Fourier transform (FT).
\end{remark}
\begin{definition}
The FRFT can also be defined as the $\alpha^{\text{th}}$ fractional power of the ordinary FT \cite{su2019analysis}, \cite{alikacsifouglu2024graph}:

\begin{equation}
\mathcal{F}^{(\alpha)}\psi_{\ell} 
= \left(e^{-j\ell\pi/2}\right)^{\alpha}\psi_{\ell} 
= e^{-j\alpha\ell\pi/2}\psi_{\ell},
\tag{11}
\label{11}
\end{equation}
where $\psi_{\ell}$, $\ell = 0,1,\ldots$, are the Hermite–Gaussian functions, which serve as the eigenfunctions of the FT with eigenvalues $e^{-j\ell\pi/2}$. The Hermite–Gaussian functions are given by:

\begin{equation}
\psi_{n}(u)
= \frac{2^{1/4}}{\sqrt{2^n n!}} H_n(\sqrt{2\pi}\,u) e^{-\pi u^2},
\tag{12}
\end{equation}
with $H_n$ being the $n^{\text{th}}$-order Hermite polynomial. This definition characterizes the FRFT through its eigenfunctions and eigenvalues, which provides a practical way to construct the FRFT kernel as a spectral expansion over the Hermite–Gaussian basis:

\begin{equation}
K_\alpha(u,u') = \sum_{\ell=0}^{\infty} \psi_\ell(u) \, e^{-j\alpha\ell\pi/2} \, \psi_\ell(u').
\tag{13}
\end{equation}

For discrete signals, the FRFT is implemented as the discrete FRFT (DFRFT) using a transformation matrix ${F}^\alpha$. Each element of the order-$\alpha$ DFRFT matrix is computed via the spectral expansion:

\begin{equation}
{F}^\alpha_{m,n} = \sum_{k=0}^{N-1} \phi_m^{(k)} \, e^{-j\alpha\pi k/2} \, \phi_n^{(k)},
\tag{14}
\end{equation}
where $\phi^{(k)}$ denotes the $k^{\text{th}}$ discrete Hermite–Gaussian vector and $\phi_n^{(k)}$ its $n^{\text{th}}$ entry. This approach provides a direct, implementable method to compute the FRFT for discrete signals in practice.
\end{definition}


\section{Manifold Fractional Harmonic Transform}
\label{sec:3}
\subsection{Definition and Properties}
\begin{definition}
The forward and inverse point cloud manifold harmonic transforms are defined as:
\begin{equation}
\hat{f} = H^T B f, \quad f = H \hat{f},
\tag{15}
\end{equation}
where \(H^T B H = I\).

Equivalently, in matrix form, the forward transform can be written as:
\begin{equation}
\hat{f} = F_M f,
\tag{16}
\end{equation}
where \(F_M := H^T B\) is the point cloud manifold harmonic matrix, and its inverse is \(F_M^{-1} = H\).

Perform the Jordan decomposition \cite{wang2017fractional} of \(F_M\):
\begin{equation}
F_M = P J P^{-1},
\tag{17}
\end{equation}
where \(J\) denotes the Jordan canonical form of \(F_M\). The point cloud manifold fractional harmonic matrix of order \(\alpha \in \mathbb{R}\) is then defined as:
\begin{equation}
F_M^{(\alpha)} := P J^\alpha P^{-1},
\tag{18}
\end{equation}
where \(J^{\alpha} = \mathrm{diag}(J_1^{\alpha},\ldots,J_N^{\alpha})\) is computed using the standard fractional power of Jordan blocks.

Thus, the point cloud manifold fractional harmonic transform of order \(\alpha \in \mathbb{R}\) is given by:
\begin{equation}
\hat{f}^{(\alpha)} = F_M^{(\alpha)} f, \quad 
f = (F_M^{(\alpha)})^{-1} \hat{f}^{(\alpha)} = F_M^{(-\alpha)} \hat{f}^{(\alpha)}.
\tag{19}
\label{eq:JORDAN_pmfht}
\end{equation}
\end{definition}
\begin{remark}
While any square matrix $F_M$ can theoretically be decomposed via the Jordan canonical form $F_M = P J P^{-1}$, this approach is often ill-conditioned in practical computations. Specifically, for point cloud data with non-uniform sampling, the Jordan blocks are extremely sensitive to small perturbations in the coordinates or the stiffness matrix $Q$, leading to significant numerical errors in the fractional power $F_M^{(\alpha)}$. 

\end{remark}
%
%
%

To address the problems of the previous definition, we present a new definition of PMFHT as follows:
\begin{definition}\label{definition:PMFHT}
To ensure the orthogonality of the transform and numerical stability, we define the normalized point cloud manifold harmonic matrix \(F_M\) as:
\begin{equation}
F_M := H^T B^{1/2},
\tag{20}
\end{equation}
where \(H\) is the matrix of generalized eigenvectors and \(B\) is the diagonal mass matrix. Since the eigenvectors satisfy the orthogonality condition \(H^T B H = I\), it is easily verified that \(F_M\) is an orthogonal matrix, i.e., \(F_M^{-1} = F_M^T = B^{1/2} H\).

Accordingly, the forward and inverse manifold harmonic transforms of a signal \(f\) can be expressed in terms of the weighted signal \(\tilde{f} = B^{1/2} f\) as:
\begin{equation}
\hat{f} = F_M (B^{1/2} f), \quad f = B^{-1/2} F_M^T \hat{f}.
\tag{21}
\end{equation}

Since \(F_M\) is an orthogonal matrix, it can be diagonalized via spectral decomposition:
\begin{equation}
F_M = V \Omega V^*,
\tag{22}
\end{equation}
where \(\Omega = \mathrm{diag}(\omega_1, \dots, \omega_N)\) contains the eigenvalues of \(F_M\) on the unit circle, and \(V\) is a unitary matrix. The point cloud manifold fractional harmonic matrix of order \(\alpha \in \mathbb{R}\) is then defined as:
\begin{equation}
F_M^{(\alpha)} := V \Omega^\alpha V^*,
\tag{23}
\end{equation}
where \(\Omega^\alpha = \mathrm{diag}(\omega_1^\alpha, \dots, \omega_N^\alpha)\). This definition bypasses the need for Jordan canonical forms, ensuring greater numerical robustness.

Thus, the point cloud manifold fractional harmonic transform (PMFHT) of order \(\alpha \in \mathbb{R}\) is given by:
\begin{equation}
\hat{f}^{(\alpha)} = F_M^{(\alpha)} (B^{1/2} f), \quad 
f = B^{-1/2} F_M^{(-\alpha)} \hat{f}^{(\alpha)}.
\tag{24}
\label{eq:PMFHT}
\end{equation}
\end{definition}
\begin{remark}
		The PMFHT performs an energy-normalized mapping of the spatial signal $f$ to yield a weighted representation:
		\begin{equation}
			\hat{f}^{(0)} =F_M^{(0)} (B^{1/2} f)=V \Omega^0 V^*(B^{1/2} f) = V V^*(B^{1/2} f) = B^{1/2} f,
			\tag{25}
		\end{equation}
		which ensures that the Euclidean $L_2$ norm in the $0^{th}$-order domain reflects the total manifold energy $\|\hat{f}^{(0)}\|_2^2 = f^T B f$.

		PMFHT Reduction to PMHT when $\alpha=1$:
		\begin{equation}
			\hat{f}^{(1)} =F_M^{(1)} (B^{1/2} f)=V \Omega^1 V^*(B^{1/2} f) = F_M(B^{1/2} f) = H^T B f,
			\tag{26}
		\end{equation}
		so the $1^{th}$-order PMFHT reduces to the standard PMHT.

\end{remark}
Definition \eqref{eq:PMFHT} is the formal definition of PMFHT, which addresses the shortcomings of the definition \eqref{eq:JORDAN_pmfht}. Fig.~\ref{fig:1} illustrates the implementation process of the PMFHT. 
\begin{figure*}[htb]
  \centering
  \includegraphics[width=1\linewidth]{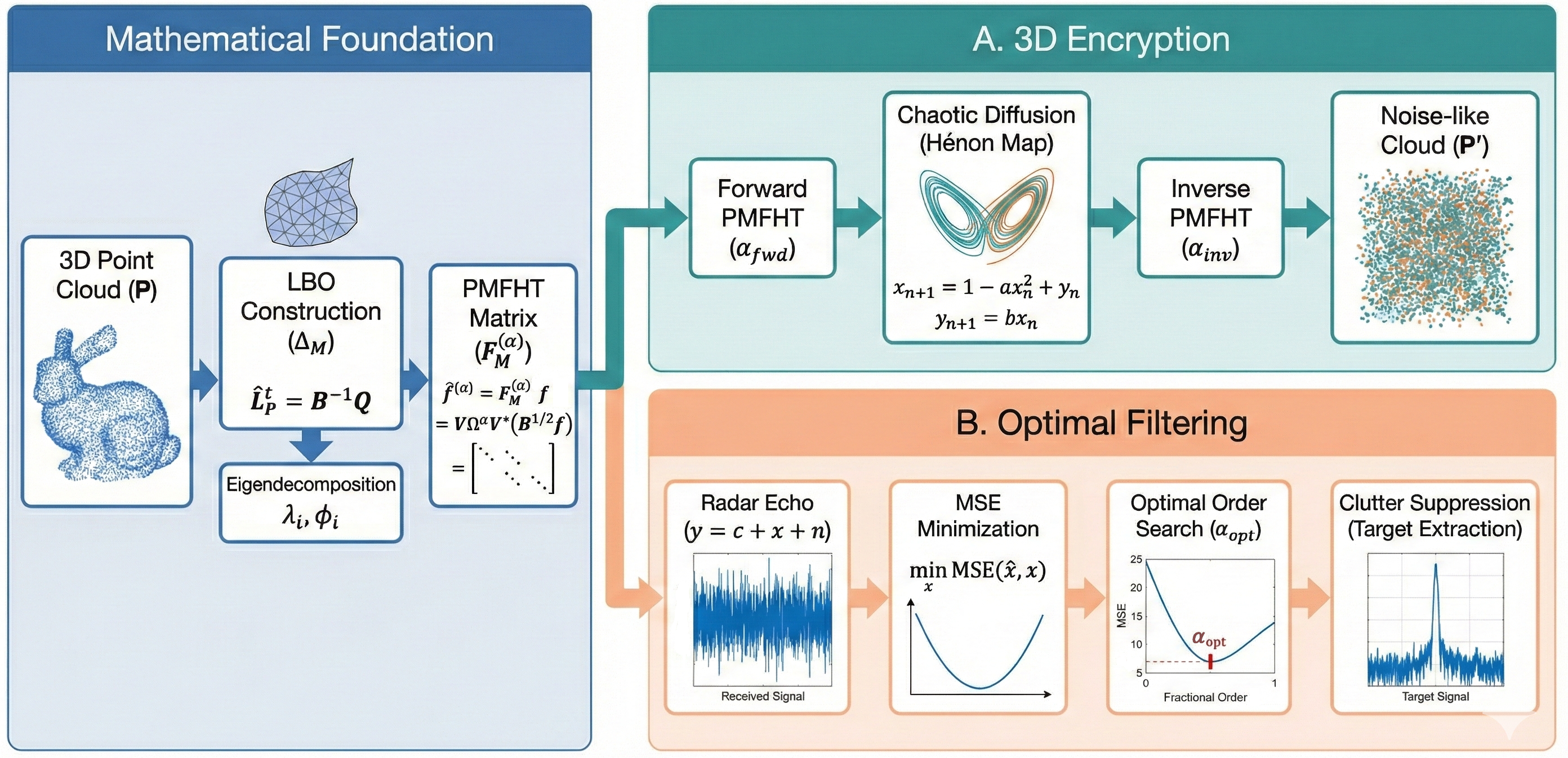}
  \caption{Schematic framework of the PMFHT: from mathematical foundation to representative applications.}
  \label{fig:1}
\end{figure*}
\begin{property}
	\text{Index additivity.} \\
	Fractional matrix powers preserve index additivity:
	\begin{equation}
		F_\alpha F_\beta
		= V \Omega_\alpha V^{-1} V \Omega_\beta V^{-1}
		= V \Omega_{\alpha+\beta} V^{-1}
		= F_{\alpha+\beta}.
		\tag{27}
	\end{equation}
\end{property}
\begin{property}
	\text{Unitarity.}\\
	For every $\alpha \in \mathbb{R}$, the PMFHT matrix $F_M^{(\alpha)}$ is unitary:
	\begin{equation}
		\bigl(F_M^{(\alpha)}\bigr)^{\!*} F_M^{(\alpha)} = I.
		\tag{28}
	\end{equation}
\end{property}

\begin{proof}
	Since $F_M = H^\top B^{1/2}$ is a real orthogonal matrix (from $H^\top BH = I$),
	its eigenvalues satisfy $|\omega_i| = 1$, so $|\omega_i^\alpha| = 1$ for all $\alpha$.
	Writing $F_M^{(\alpha)} = V\Omega^\alpha V^*$ with unitary $V$:
	\begin{equation}
		\bigl(F_M^{(\alpha)}\bigr)^{\!*} F_M^{(\alpha)}
		= V(\Omega^\alpha)^*\Omega^\alpha V^*
		= V\,I\,V^* = I.
		\tag{29}
	\end{equation}
\end{proof}

\begin{remark}
	A direct consequence is that the inverse PMFHT equals the conjugate transpose
	of the forward PMFHT:
	\begin{equation}
		F_M^{(-\alpha)} = \bigl(F_M^{(\alpha)}\bigr)^{\!*}.
		\tag{30}
	\end{equation}
\end{remark}
%
\begin{property}
	\text{Periodicity.}\\
	The PMFHT is periodic in the fractional order if and only if all eigenvalues
	$\{\omega_i\}$ of $F_M$ are roots of unity.
	Specifically, if there exists a positive integer $k$ such that
	$\omega_i^k = 1$ for all $i$, then
	\begin{equation}
		F_M^{(\alpha+k)} = F_M^{(\alpha)} \quad \forall\,\alpha\in\mathbb{R}.
		\tag{31}
	\end{equation}
\end{property}

\begin{proof}
	By index additivity and $\Omega^k = I$:
	\begin{equation}
		F_M^{(\alpha+k)}
		= F_M^{(\alpha)} F_M^{(k)}
		= V\Omega^\alpha V^* \cdot V\Omega^k V^*
		= V\Omega^\alpha V^*
		= F_M^{(\alpha)}.
		\tag{32}
	\end{equation}
\end{proof}

\begin{remark}
	For a general Riemannian manifold the eigenvalues of $F_M$ need not be roots
	of unity, so periodicity in $\alpha$ does not hold in general.
	On symmetric manifolds whose LBO spectrum yields a circulant $F_M$, for example, uniformly sampled points on a circle, the eigenvalues are $N$-th roots
	of unity and the period $k = 4$ matches the classical FRFT.
\end{remark}
\begin{theorem}[Parseval's Theorem]
	Let $f, g$ be signals on the point cloud and define the manifold inner product
	$\langle f,g\rangle_B := f^\top Bg$.
	For any $\alpha \in \mathbb{R}$,
	\begin{equation}
		\bigl\langle \hat{f}^{(\alpha)},\,\hat{g}^{(\alpha)}\bigr\rangle
		\;=\; \langle f,\,g\rangle_B,
		\tag{33}
	\end{equation}
	and in particular,
	\begin{equation}
		\bigl\|\hat{f}^{(\alpha)}\bigr\|_2^2
		\;=\; f^\top Bf.
		\tag{34}
		\label{eq:energy}
	\end{equation}
\end{theorem}

\begin{proof}
	Using $\hat{f}^{(\alpha)} = F_M^{(\alpha)}(B^{1/2}f)$,
	$\hat{g}^{(\alpha)} = F_M^{(\alpha)}(B^{1/2}g)$:
	\begin{align}
		\bigl\langle \hat{f}^{(\alpha)},\hat{g}^{(\alpha)}\bigr\rangle
		&= \bigl(B^{1/2}f\bigr)^\top
		\bigl(F_M^{(\alpha)}\bigr)^{\!*} F_M^{(\alpha)}
		\bigl(B^{1/2}g\bigr) \notag\\
		&= \bigl(B^{1/2}f\bigr)^\top\!\bigl(B^{1/2}g\bigr)
		= f^\top Bg
		= \langle f,g\rangle_B.
		\tag{35}
	\end{align}
	Setting $f = g$ yields~\eqref{eq:energy}.
\end{proof}
\begin{theorem}[Boundedness Theorem]
	\label{thm:bounded}
	Let $f \in \mathbb{R}^N$ be a signal on the point cloud. Then for every
	fractional order $\alpha \in \mathbb{R}$, the PMFHT satisfies
	\begin{equation}
		\bigl\|\hat{f}^{(\alpha)}\bigr\|_2 \;=\; \bigl\|B^{1/2}f\bigr\|_2
		\;\leq\; \bigl\|B^{1/2}\bigr\|_2\,\bigl\|f\bigr\|_2.
		\tag{36}
	\end{equation}
Thus, the PMFHT is a bounded linear operator. Furthermore, the equality $\|\hat{f}^{(\alpha)}\|_2 = \|B^{1/2}f\|_2 = \sqrt{f^\top B f} = \|f\|_B$ indicates that it acts as an isometry from the spatial domain to the fractional spectral domain.
\end{theorem}

\begin{proof}
	This result follows directly from the unitarity of $F_M^{(\alpha)}$ and the sub-multiplicativity of the matrix 2-norm.
\end{proof}
Fig.~\ref{fig:2} is a visualization of the manifold harmonic bases of the horse point cloud data.
Figs.~\ref{fig:5} illustrate the PMFHT spectra of the horse point cloud data under different fractional orders (0.2, 0.4, 0.6, 0.8 and 1.0). Since the point cloud signal is represented in \(\mathbb{R}^{N \times 3}\), applying PMFHT produces three separate spectra corresponding to the \(x\), \(y\), and \(z\)-coordinate components. In addition, the figures also present the fused energy spectrum obtained by computing the \(\ell_2\)-norm across the three component spectra.

%
%
\begin{algorithm}[htbp]
	\caption{Normalized Orthogonal PMFHT}\label{alg:PMFHT_normalized_orthogonal}
	\begin{algorithmic}[1]
		\State \textbf{Input:} Point cloud signal $f$, fractional order $\alpha \in \mathbb{R}$
		\State \textbf{Output:} Fractional spectral coefficients $\hat{f}^{(\alpha)}$
		
		\State Calculate discrete LBO matrix $\hat{L}_{P}^{t} = B^{-1}Q$ via local Voronoi estimation
		\State Solve generalized eigenproblem $QH = -\lambda BH$ to obtain eigenvectors $H$
		\State Construct normalized manifold Fourier matrix: $F_M = H^T B^{1/2}$
		\State Perform spectral decomposition of orthogonal matrix $F_M$: $F_M = V \Omega V^*$ 
		\State Compute fractional power of eigenvalues: $\Omega^\alpha = \operatorname{diag}(\omega_1^\alpha, \omega_2^\alpha, \ldots, \omega_N^\alpha)$
		\State Construct manifold fractional harmonic matrix: $F_M^{(\alpha)} = V \Omega^\alpha V^*$
		\State Apply transform with weighted signal: $\hat{f}^{(\alpha)} = F_M^{(\alpha)} (B^{1/2} f)$
		
	\end{algorithmic}
\end{algorithm}

\begin{figure}[t]
	\centering
	\subfloat[(a)]{\includegraphics[width=0.18\columnwidth]{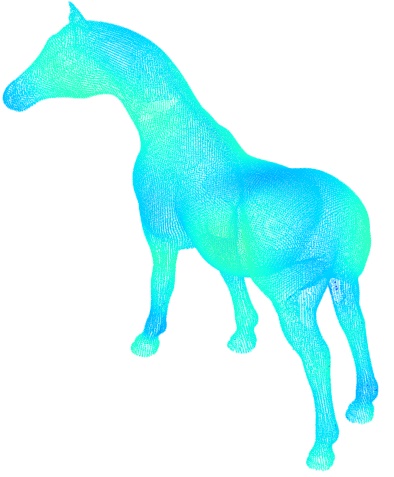}}
	\hfill
	\subfloat[(b)]{\includegraphics[width=0.18\columnwidth]{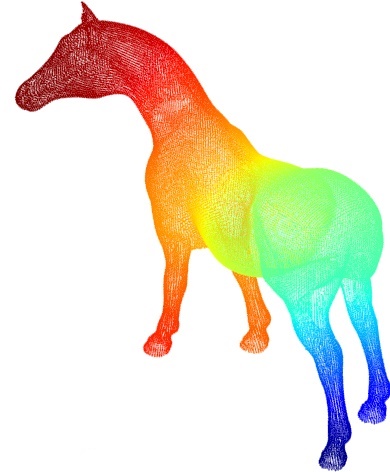}}
	\hfill
	\subfloat[(c)]{\includegraphics[width=0.18\columnwidth]{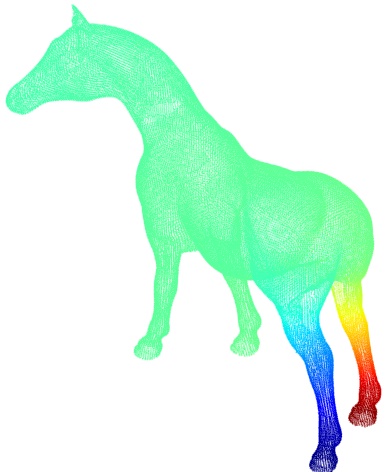}}
	\hfill
	\subfloat[(d)]{\includegraphics[width=0.18\columnwidth]{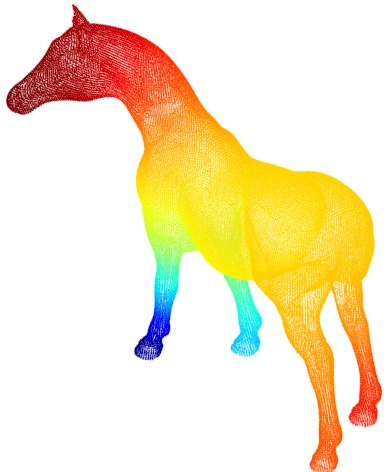}}
	\caption{Visualization of manifold harmonic bases for the horse point cloud data. (a) Manifold harmonic basis 1. (b) Manifold harmonic basis 2. (c) Manifold harmonic basis 3. (d) Manifold harmonic basis 4.}
	\label{fig:2}
\end{figure}

\begin{figure*}[htb]
  \centering
  \includegraphics[width=1\linewidth]{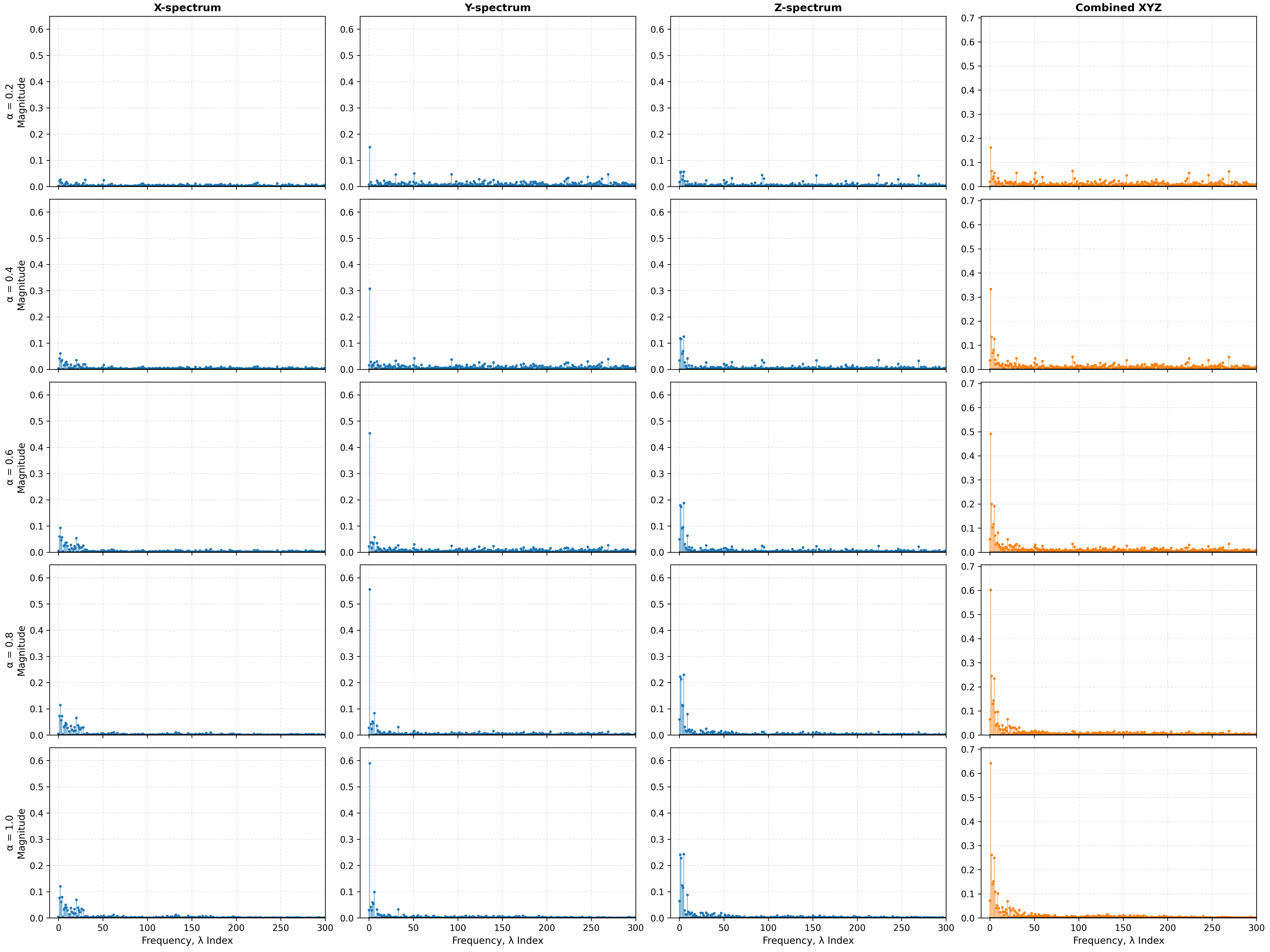}
  \caption{PMFHT spectra of horse point cloud ($\alpha = 0.2, 0.4, 0.6, 0.8, 1.0$).}
  \label{fig:5}
\end{figure*}
\subsection{Convolution and Correlation}
%
%
\begin{definition}[Generalized translation]
	The translation of a continuous signal $f(t)$ by a shift $i$ is fundamentally equivalent to convolving the signal with a Dirac delta function $\delta_i(t)$ centered at $i$:
	\begin{equation}
		(T_i f)(t) = (f * \delta_i)(t) = \int_{\mathbb{R}} \hat{f}(\xi) \hat{\delta}_i(\xi) e^{2\pi i \xi t} d\xi.
		\tag{37}
	\end{equation}
	Analogously, on a point cloud manifold, we define the $\alpha$-order generalized translation of a signal $f$ toward point $i$ as its manifold fractional convolution with the unit impulse $\delta_i \in \mathbb{R}^N$:
	\begin{equation}
		T_i^{(\alpha)} f \;:=\; f \circledast_\alpha \delta_i.
		\tag{38}
	\end{equation}
	The value of the translated signal at point $n$ is computed as the inverse transform of the spectral product over all fractional frequencies $\ell$:
	\begin{equation}
		(T_i^{(\alpha)} f)(n) = \sum_{\ell=0}^{N-1} \hat{f}^{(\alpha)}(\ell) \hat{\delta}_i^{(\alpha)}(\ell) \left[ B^{-1/2} F_M^{(-\alpha)} \right]_{n, \ell},
		\tag{39}
	\end{equation}
	where $\hat{\delta}_i^{(\alpha)}$ is the fractional spectrum of the impulse.\\
	According to the definition~\eqref{eq:PMFHT}:\\
	\begin{equation}
		\hat{\delta}_i^{(\alpha)} = F_M^{(\alpha)} (B^{1/2} \delta_i)=\sqrt{B_{ii}} F_{M, i}^{(\alpha)}, 
		\tag{40}
	\end{equation}
	where $F_{M, i}^{(\alpha)}$ denotes the $i^{th}$ column of $F_M^{(\alpha)}$. 
	
	Substituting this into the summation yields:
	\begin{equation}
		(T_i^{(\alpha)} f)(n) = \sum_{\ell=0}^{N-1} \left( \hat{f}^{(\alpha)}(\ell) \cdot \sqrt{B_{ii}} \left[ F_M^{(\alpha)} \right]_{\ell, i} \right) \left[ B^{-1/2} F_M^{(-\alpha)} \right]_{n, \ell}.
		\tag{41}
	\end{equation}
	
	Converting this element-wise summation over the frequency index $\ell$ back into compact matrix-vector notation, we obtain the global spatial representation of the translated signal:
	\begin{equation}
		T_i^{(\alpha)} f = B^{-1/2} F_M^{(-\alpha)} \bigl( \hat{f}^{(\alpha)} \odot \hat{\delta}_i^{(\alpha)} \bigr).
		\tag{42}
	\end{equation}
\end{definition}
\begin{definition}[Manifold Fractional Convolution]
	\label{def:conv}
	Let $f, g \in \mathbb{R}^N$ be signals on the point cloud manifold. Utilizing the generalized translation operator, we define the {$\alpha$-order manifold convolution} of $f$ and $g$ as the spatial linear combination of translated signals:
	\begin{equation}
		(f \circledast_\alpha g)
		\;:=\;
		\sum_{i=1}^N f(i) ({T}_i^{(\alpha)} g).
		\tag{43}
	\end{equation}
\end{definition}

\begin{theorem}[Convolution Theorem]
	\label{thm:conv}
	The PMFHT of the $\alpha$-order manifold convolution satisfies the spectral product formulation:
	\begin{equation}
		\widehat{(f \circledast_\alpha g)}^{(\alpha)}
		\;=\;
		\hat{f}^{(\alpha)} \odot \hat{g}^{(\alpha)},
		\tag{44}
	\end{equation}
	where $\odot$ denotes element-wise (Hadamard) multiplication.
\end{theorem}

\begin{proof}
	Applying the forward PMFHT operator $F_M^{(\alpha)} B^{1/2}$ to Definition \ref{def:conv}, and exploiting the linearity of the transform, we obtain:
	\begin{align}
		\widehat{(f \circledast_\alpha g)}^{(\alpha)}
		&= F_M^{(\alpha)} B^{1/2} \left( \sum_{i=1}^N \sqrt{B_{ii}} f(i) \Bigl( B^{-1/2} F_M^{(-\alpha)} \bigl( \hat{g}^{(\alpha)} \odot F_{M, i}^{(\alpha)} \bigr) \Bigr) \right) \notag\\
		&= \sum_{i=1}^N \sqrt{B_{ii}} f(i) F_M^{(\alpha)} F_M^{(-\alpha)} \bigl( \hat{g}^{(\alpha)} \odot F_{M, i}^{(\alpha)} \bigr). 
		\notag\\
		&= \sum_{i=1}^N \sqrt{B_{ii}} f(i) \bigl( \hat{g}^{(\alpha)} \odot F_{M, i}^{(\alpha)} \bigr). 
		\tag{45}
	\end{align}
	Because the Hadamard product distributes over addition, we can factor out $\hat{g}^{(\alpha)}$:
	\begin{align}
		\widehat{(f \circledast_\alpha g)}^{(\alpha)}
		&= \hat{g}^{(\alpha)} \odot \left( \sum_{i=1}^N \sqrt{B_{ii}} f(i) F_{M, i}^{(\alpha)} \right). 
		\tag{46}
	\end{align}
	Observe the summation term. Since $F_{M, i}^{(\alpha)}$ is the $i^{th}$ column of $F_M^{(\alpha)}$, this summation is precisely the matrix-vector expansion of the forward PMFHT of $f$:
	\begin{equation}
		\sum_{i=1}^N F_{M, i}^{(\alpha)} \bigl( B^{1/2} f \bigr)_i \;=\; F_M^{(\alpha)} (B^{1/2} f) \;=\; \hat{f}^{(\alpha)}. 
		\tag{47}
	\end{equation}
	Substituting this back yields the final result:
	\begin{equation}
		\widehat{(f \circledast_\alpha g)}^{(\alpha)} \;=\; \hat{g}^{(\alpha)} \odot \hat{f}^{(\alpha)} \;=\; \hat{f}^{(\alpha)} \odot \hat{g}^{(\alpha)}. 
		\tag{48}
	\end{equation}
	This completes the proof.
\end{proof}
%

\begin{remark}
	The convolution $f \circledast_\alpha g$ is the unique signal whose
	$\alpha$-order PMFHT spectrum equals the pointwise product of the individual
	spectra.  For $\alpha = 1$ this reduces to the standard spectral-domain
	product on the manifold harmonics.  The operator is commutative
	($f \circledast_\alpha g = g \circledast_\alpha f$) because the Hadamard
	product is commutative.
\end{remark}
In classical Fourier analysis, translating a signal's spectrum by a frequency $\nu$ is achieved by modulating the spatial signal with a complex exponential harmonic:
\begin{equation}
	\hat{g}(\xi - \nu) = \mathcal{F} \{ g(t) e^{2\pi i \nu t} \}.
	\tag{49}
\end{equation}
On a point cloud manifold, the classical harmonic $e^{2\pi i \nu t}$ is generalized to the $k^{th}$ basis vector of the inverse PMFHT in the spatial domain. According to the inverse transform $f = B^{-1/2} F_M^{(-\alpha)} \hat{f}^{(\alpha)}$, the true $k^{th}$ spatial fractional harmonic basis is scaled by the mass matrix, given by $B^{-1/2} F_{M, k}^{(-\alpha)}$. 

To translate the spectrum of a signal $g$, we first modulate $g$ with this spatial basis, yielding a modulated signal $(B^{-1/2} F_{M, k}^{(-\alpha)}) \odot g$. We then apply the forward PMFHT to this modulated signal:
\begin{equation*}
	F_M^{(\alpha)} B^{1/2} \Bigl( (B^{-1/2} F_{M, k}^{(-\alpha)}) \odot g \Bigr) = F_M^{(\alpha)} \bigl( F_{M, k}^{(-\alpha)} \odot g \bigr).
	\tag{50}
\end{equation*}

\begin{definition}[Generalized Spectral Translation]
	\label{def:spectral_trans}
	We define the $\alpha$-order generalized spectral translation of the spectrum $\hat{g}^{(\alpha)}$ by frequency index $k$ as the forward PMFHT of the spatially modulated signal:
	\begin{equation}
		(\mathcal{S}_k^{(\alpha)} \hat{g}^{(\alpha)}) \;:=\; F_M^{(\alpha)} \bigl( F_{M, k}^{(-\alpha)} \odot g \bigr).
		\tag{51}
	\end{equation}
\end{definition}
\begin{definition}[Spectral Fractional Convolution]
	\label{def:dual_conv}
	Analogous to spatial convolution, the $\alpha$-order spectral convolution of $\hat{f}^{(\alpha)}$ and $\hat{g}^{(\alpha)}$ is defined as the spectral linear combination of translated spectra:
	\begin{equation}
		\bigl(\hat{f}^{(\alpha)} \circledast_\alpha^* \hat{g}^{(\alpha)}\bigr)
		\;:=\;
		\sum_{k=0}^{N-1} \hat{f}^{(\alpha)}(k) (\mathcal{S}_k^{(\alpha)} \hat{g}^{(\alpha)}).
		\tag{52}
	\end{equation}
\end{definition}

\begin{theorem}[Dual Convolution Theorem]
	\label{thm:dual_conv}
	The $\alpha$-order spectral convolution corresponds to the PMFHT of the element-wise spatial product of the signals:
	\begin{equation}
		\hat{f}^{(\alpha)} \circledast_\alpha^* \hat{g}^{(\alpha)}
		\;=\;
		\widehat{(f \odot g)}^{(\alpha)}.
		\tag{53}
	\end{equation}
\end{theorem}

\begin{proof}
	By the inverse PMFHT, the signal $f = B^{-1/2} F_M^{(-\alpha)} \hat{f}^{(\alpha)} = \sum_{k=0}^{N-1} \hat{f}^{(\alpha)}(k) \bigl( B^{-1/2} F_{M, k}^{(-\alpha)} \bigr)$.
	Substituting this expansion into the spatial product $(f \odot g)$ and noting that $B^{-1/2}$ is a diagonal matrix:
	\begin{equation}
		(f \odot g) = \left( \sum_{k=0}^{N-1} \hat{f}^{(\alpha)}(k) B^{-1/2} F_{M, k}^{(-\alpha)} \right) \odot g = B^{-1/2} \sum_{k=0}^{N-1} \hat{f}^{(\alpha)}(k) \bigl( F_{M, k}^{(-\alpha)} \odot g \bigr).
		\tag{54}
	\end{equation}
	Applying the forward PMFHT operator $F_M^{(\alpha)} B^{1/2}$ to $(f \odot g)$ yields:
	\begin{align}
		\widehat{(f \odot g)}^{(\alpha)}
		&= F_M^{(\alpha)} B^{1/2} \left( B^{-1/2} \sum_{k=0}^{N-1} \hat{f}^{(\alpha)}(k) \bigl( F_{M, k}^{(-\alpha)} \odot g \bigr) \right) \notag\\
		&= \sum_{k=0}^{N-1} \hat{f}^{(\alpha)}(k) F_M^{(\alpha)} \bigl( F_{M, k}^{(-\alpha)} \odot g \bigr).
		\tag{55}
	\end{align}
	Recognizing the term inside the summation as the generalized spectral translation $(\mathcal{S}_k^{(\alpha)} \hat{g}^{(\alpha)})$, we obtain:
	\begin{equation}
		\widehat{(f \odot g)}^{(\alpha)} = \sum_{k=0}^{N-1} \hat{f}^{(\alpha)}(k) (\mathcal{S}_k^{(\alpha)} \hat{g}^{(\alpha)}) = \hat{f}^{(\alpha)} \circledast_\alpha^* \hat{g}^{(\alpha)}.
		\tag{56}
	\end{equation}
	This completes the proof.
\end{proof}

\begin{definition}[Manifold Fractional Cross-Correlation]
	\label{def:corr}
	The {$\alpha$-order manifold cross-correlation} of $f$ and $g$ is defined as the spatial linear combination of the conjugate-translated signal $g$, weighted by the conjugated signal $f$:
	\begin{equation}
		(f \star_\alpha g)
		\;:=\;
		\sum_{i=1}^N \overline{f(i)} \bigl({T}_{-i}^{(\alpha)} g\bigr),
		\tag{57}
	\end{equation}
	where $\overline{(\cdot)}$ denotes the complex conjugate.
\end{definition}

\begin{theorem}[Correlation Theorem]
	\label{thm:corr}
	The PMFHT of the $\alpha$-order manifold cross-correlation satisfies:
	\begin{equation}
		\widehat{(f \star_\alpha g)}^{(\alpha)}
		\;=\;
		\overline{\hat{f}^{(\alpha)}} \odot \hat{g}^{(\alpha)}.
		\tag{58}
	\end{equation}
\end{theorem}

\begin{proof}
	Applying the forward PMFHT operator $F_M^{(\alpha)} B^{1/2}$ to Definition \ref{def:corr}, and exploiting the linearity of the transform, we obtain:
	\begin{align}
		\widehat{(f \star_\alpha g)}^{(\alpha)}
		&= F_M^{(\alpha)} B^{1/2} \left( \sum_{i=1}^N \overline{f(i)} \Bigl( B^{-1/2} F_M^{(-\alpha)} \bigl( \hat{g}^{(\alpha)} \odot \overline{\sqrt{B_{ii}}}\overline{F_{M, i}^{(\alpha)}} \bigr) \Bigr) \right) \notag\\
		&= \sum_{i=1}^N \overline{f(i)} F_M^{(\alpha)} F_M^{(-\alpha)} \bigl( \hat{g}^{(\alpha)} \odot \overline{\sqrt{B_{ii}}}\overline{F_{M, i}^{(\alpha)}} \bigr).
		\notag\\
		&= \sum_{i=1}^N \overline{f(i)} \bigl( \hat{g}^{(\alpha)} \odot \overline{\sqrt{B_{ii}}}\overline{F_{M, i}^{(\alpha)}} \bigr).
		\tag{59}
	\end{align}
	Because the Hadamard product distributes over addition, we factor out $\hat{g}^{(\alpha)}$:
	\begin{equation}
		\widehat{(f \star_\alpha g)}^{(\alpha)}
		= \hat{g}^{(\alpha)} \odot \left( \sum_{i=1}^N \overline{f(i)} \overline{\sqrt{B_{ii}}} \overline{F_{M, i}^{(\alpha)}} \right).
		\tag{60}
	\end{equation}
	Observe the summation term. It is exactly the complex conjugate of the matrix-vector expansion for the forward PMFHT of $f$:
	\begin{equation}
		\sum_{i=1}^N \overline{f(i)} \overline{\sqrt{B_{ii}}} \overline{F_{M, i}^{(\alpha)}}
		\;=\; \overline{ \sum_{i=1}^N F_{M, i}^{(\alpha)} \bigl( B^{1/2} f \bigr)_i }
		\;=\; \overline{ \hat{f}^{(\alpha)} }.
		\tag{61}
	\end{equation}
	Substituting this back yields the final result:
	\begin{equation}
		\widehat{(f \star_\alpha g)}^{(\alpha)} \;=\; \hat{g}^{(\alpha)} \odot \overline{\hat{f}^{(\alpha)}} \;=\; \overline{\hat{f}^{(\alpha)}} \odot \hat{g}^{(\alpha)}.
		\tag{62}
	\end{equation}
	This completes the proof.
\end{proof}
\subsection{Sampling and Interpolation}

Consider a point cloud signal 
$f \in \mathbb{R}^{N \times 3}$, where each row corresponds to a 3D point (or a 3-channel feature). Sampling selects $K$ points ($K < N$) to form a sampled signal $f_{\mathcal{K}} \in \mathbb{R}^{K \times 3}$, where 
$\mathcal{K} = \{\mathcal{K}_0, \ldots, \mathcal{K}_{K-1}\}$ denotes the sampled point indices. The sampling operator $\Psi \in \mathbb{R}^{K \times N}$ is defined as
\begin{equation}
	\Psi_{i,j}=
	\begin{cases}
		1, & j = \mathcal{K}_i,\\
		0, & \text{otherwise},
	\end{cases}
	\tag{63}
\end{equation}
such that
\begin{equation}
	\text{sampling: } \quad f_{\mathcal{K}} = \Psi f \in \mathbb{R}^{K \times 3}.\tag{64}
\end{equation}

The interpolation operator $\Phi \in \mathbb{R}^{N \times K}$ maps the sampled signal back to the full point set:
\begin{equation}
	\text{interpolation: } \quad f' = \Phi f_{\mathcal{K}} = \Phi \Psi f \in \mathbb{R}^{N \times 3},
	\tag{65}
\end{equation}
where $f'$ is either an approximation or an exact recovery of $f$.


\begin{definition}[$\alpha$-Bandlimited Signal]
	A point cloud signal $f$ is said to be $\alpha$-bandlimited if there exists an integer $K_b \in \{1, \ldots, N\}$ such that its PMFHT spectrum satisfies
	\begin{equation}
		\hat{f}^{(\alpha)}(k) = \mathbf{0} \quad \text{for all } k \ge K_b,
		\tag{66}
	\end{equation}
	where $\hat{f}^{(\alpha)} = F_M^{(\alpha)} B^{1/2} f$.  
	The smallest such $K_b$ is referred to as the $\alpha$-bandwidth of $f$. We denote the set of all such $\alpha$-bandlimited signals by $\operatorname{BL}_{K_b}(F_M^{(\alpha)})$.
\end{definition} 

To facilitate the derivation, we define the \emph{bandlimited manifold basis matrix} $\mathcal{V}_{K_b}^{(\alpha)} \in \mathbb{R}^{N \times K_b}$ as:
\begin{equation}
	\mathcal{V}_{K_b}^{(\alpha)} \;:=\; B^{-1/2} (F_M^{(-\alpha)})_{(K_b)},
	\tag{67}
\end{equation}
where $(F_M^{(-\alpha)})_{(K_b)}$ denotes the submatrix containing the first $K_b$ columns of $F_M^{(-\alpha)}$. By the inverse PMFHT, any $f \in \operatorname{BL}_{K_b}(F_M^{(\alpha)})$ can be exactly synthesized as $f = \mathcal{V}_{K_b}^{(\alpha)} \hat{f}^{(\alpha)}_{(K_b)}$.

\begin{theorem}[Sampling Theorem for PMFHT] 
	\label{thm:sampling}
	For any $\alpha$-bandlimited point cloud signal $f \in \operatorname{BL}_{K_b}(F_M^{(\alpha)})$, perfect reconstruction from $K \ge K_b$ sampled points is achievable if and only if the sampling operator $\Psi$ satisfies the full column rank condition:
	\begin{equation}
		\operatorname{rank}\!\left(\Psi \mathcal{V}_{K_b}^{(\alpha)}\right) = K_b.
		\tag{68}
	\end{equation}
	In this case, perfect recovery $f = \Phi \Psi f$ holds by choosing the interpolation operator
	\begin{equation}
		\Phi = \mathcal{V}_{K_b}^{(\alpha)} U,
		\tag{69}
	\end{equation}
	where the matrix $U \in \mathbb{R}^{K_b \times K}$ is a left inverse satisfying $U \Psi \mathcal{V}_{K_b}^{(\alpha)} = I_{K_b \times K_b}$.
\end{theorem} 

\begin{proof}  
	Since $f \in \operatorname{BL}_{K_b}(F_M^{(\alpha)})$, its spatial representation is purely governed by the bandlimited basis: $f = \mathcal{V}_{K_b}^{(\alpha)} \hat{f}^{(\alpha)}_{(K_b)}$, where $\hat{f}^{(\alpha)}_{(K_b)} \in \mathbb{R}^{K_b \times 3}$ contains the nonzero spectral coefficients. Applying the sampling operator yields:
	\begin{equation}
		f_{\mathcal{K}} = \Psi f = \Psi \mathcal{V}_{K_b}^{(\alpha)} \hat{f}^{(\alpha)}_{(K_b)}.\tag{70}
	\end{equation}
	Perfect recovery requires uniquely solving this linear system for the true spectrum $\hat{f}^{(\alpha)}_{(K_b)}$. This is algebraically possible if and only if the $K \times K_b$ matrix $\Psi \mathcal{V}_{K_b}^{(\alpha)}$ has full column rank $K_b$. Under this condition, a left inverse $U$ exists such that $U \Psi \mathcal{V}_{K_b}^{(\alpha)} = I_{K_b \times K_b}$, giving:
	\begin{equation}
		\hat{f}^{(\alpha)}_{(K_b)} = U f_{\mathcal{K}}.\tag{71}
	\end{equation}
	Substituting this back into the synthesis equation recovers the full signal:
	\begin{equation}
		f = \mathcal{V}_{K_b}^{(\alpha)} U f_{\mathcal{K}} = \Phi f_{\mathcal{K}},\tag{72}
	\end{equation}
	which proves that $f = \Phi \Psi f$. 
\end{proof} 

The previous section demonstrates that a qualified sampling operator, obtained by selecting $K$ points such that the corresponding rows in $\mathcal{V}_{K_b}^{(\alpha)}$ are linearly independent, guarantees perfect recovery for bandlimited signals. However, the choice of such points is not unique. In this section, we design an optimal sampling strategy to achieve maximum robustness against noise.

Consider additive noise $e \in \mathbb{R}^{K \times 3}$ introduced during the sampling process on the point cloud:
\begin{equation}
	f_{\mathcal{K}} = \Psi f + e.\tag{73}
\end{equation}
Using a qualified interpolation operator $\Phi$, the recovered signal $f^{\prime}_{e}$ becomes:
\begin{equation}
	f^{\prime}_{e} = \Phi f_{\mathcal{K}} = \Phi(\Psi f + e) = f + \Phi e.\tag{74}
\end{equation}
The reconstruction error is strictly bounded by the spectral norm of the operators:
\begin{align}
	\|f^{\prime}_{e} - f\|_{2} &= \|\Phi e\|_{2} = \|\mathcal{V}_{K_b}^{(\alpha)} U e\|_{2} \notag \\
	&\leq \|\mathcal{V}_{K_b}^{(\alpha)}\|_{2} \|U\|_{2} \|e\|_{2}.\tag{75}
\end{align}
Since the manifold geometry $\|\mathcal{V}_{K_b}^{(\alpha)}\|_{2}$ and the noise level $\|e\|_{2}$ are intrinsically fixed, minimizing the upper bound of the reconstruction error is equivalent to minimizing the spectral norm $\|U\|_{2}$. Because $U$ is the Moore-Penrose pseudo-inverse of $\Psi \mathcal{V}_{K_b}^{(\alpha)}$, minimizing $\|U\|_{2}$ is mathematically equivalent to maximizing the smallest singular value $\sigma_{\min}$ of $\Psi \mathcal{V}_{K_b}^{(\alpha)}$. Thus, the optimal sampling strategy is formulated as:
\begin{equation}
	\Psi^{\mathrm{opt}} = \arg \max_{\Psi} \sigma_{\min}\left(\Psi \mathcal{V}_{K_b}^{(\alpha)}\right). \tag{76}
\end{equation}
\section{Applications}
\label{sec:4}
\subsection{3D Point Cloud Encryption in the PMFHT Domain}
The proposed encryption scheme leverages the unique properties of the PMFHT and chaotic dynamics to achieve high-security 3D data protection. The core idea is to project the spatial coordinates of the point cloud into a fractional manifold spectral domain, where the geometric features are scrambled using chaotic phase masks before being reconstructed back into a noise-like spatial distribution.

The encryption process is designed as a multi-stage pipeline consisting of three primary phases: multi-order forward transformation, chaotic spectral diffusion, and multi-order inverse transformation. Let the input point cloud be denoted as $P = \{\mathbf{p}_i\}_{i=1}^N$, where each point $\mathbf{p}_i = [x_i, y_i, z_i]^T \in \mathbb{R}^3$. The three spatial dimensions are treated as independent signals defined over the manifold surface ${M}$.

Unlike traditional global transformations, we utilize independent fractional orders for each dimension to expand the key space. Let $\alpha_{fwd} = [p_{x,1}, p_{y,1}, p_{z,1}]$ be the set of forward fractional orders. First, the manifold harmonic basis ${H}$ and the mass matrix ${B}$ are computed via the discrete LBO. The unitary-like transformation operator ${F}_M$ is defined as
${F}_M = {H}^T {B}^{1/2}$.
By performing the eigen-decomposition ${F}_M = {V} {\Omega} {V}^H$, the fractional manifold spectral coefficients $\hat{\mathbf{S}}_{dim}$ for each dimension $dim \in \{x, y, z\}$ are obtained:
\begin{equation}
    \hat{{S}}_{dim} = {V} {\Omega}^{p_{dim,1}} {V}^H ({B}^{1/2} {f}_{dim})
    \tag{77}
\end{equation}
where ${f}_{dim}$ represents the vector of coordinates for the corresponding dimension.

To achieve high sensitivity and non-linearity, a chaotic system is employed to diffuse the spectral coefficients. We adopt the Hénon map as the pseudo-random generator, defined by the following iterative equations:
\begin{equation}
    \begin{cases}
    u_{n+1} = 1 - a u_n^2 + v_n \\
    v_{n+1} = b u_n
    \end{cases}
    \tag{78}
\end{equation}
where $(a, b)$ are system parameters and $(u_0, v_0)$ serve as initial keys. The generated sequence is normalized and mapped to a phase mask $\Phi \in \mathbb{C}^N$ such that $\Phi_n = \exp(j 2\pi \psi_n)$. The scrambled spectral coefficients $\hat{{S}}'_{dim}$ are calculated via element-wise complex multiplication:
\begin{equation}
    \hat{{S}}'_{dim} = \hat{{S}}_{dim} \odot \Phi_{dim}
    \tag{79}
\end{equation}
This operation randomizes the phase distribution of the manifold harmonics while preserving the spectral energy, effectively masking the underlying geometry.

The final encrypted point cloud $P'$ is reconstructed by applying the inverse PMFHT with a second set of independent orders $\alpha_{inv} = [p_{x,2}, p_{y,2}, p_{z,2}]$. This "double-lock" mechanism ensures that even if one set of orders is compromised, the data remains protected. The encrypted coordinates are given by:
\begin{equation}
    {f}'_{dim} = {B}^{-1/2} \left( {V} {\Omega}^{-p_{dim,2}} {V}^H \hat{{S}}'_{dim} \right)
    \tag{80}
\end{equation}
The resulting $P' = [{f}'_x, {f}'_y, {f}'_z]$ exhibits a noise-like spatial distribution with no discernible geometric features.


\begin{algorithm}[htbp]
	
	\caption{Multi-order PMFHT and Chaotic Encryption}
	\label{alg:encryption_compact}
	\begin{algorithmic}[1]
		\State \textbf{Input:} Point cloud $P \in \mathbb{R}^{N \times 3}$, keys $K = \{\alpha_{fwd}, \alpha_{inv}, u_0, v_0\}$
		\State \textbf{Output:} Encrypted point cloud $P'$
		
		\State Compute LBO basis ${H}$ and mass matrix ${B}$ 
		\State Construct transformation operator: ${F}_M = {H}^T {B}^{1/2}$
		\State Spectral decomposition: ${F}_M = {V} {\Omega} {V}^H$
		\State {For each} dimension $d \in \{x, y, z\}$:
		\State \hspace{0.5cm} a) Forward fractional transform:
		$\hat{{S}}_d = {V} {\Omega}^{p_{d,1}} {V}^H ({B}^{1/2} {f}_d)$
		\State \hspace{0.5cm} b) Generate chaotic phase mask $\Phi_d$ using Hénon map$(u_0, v_0)$
		\State \hspace{0.5cm} c) Spectral chaotic diffusion:
		$\hat{{S}}'_d = \hat{{S}}_d \odot \exp(j 2\pi \Phi_d)$
		\State \hspace{0.5cm} d) Inverse fractional transform (Secondary order):
		${f}'_d = {B}^{-1/2} \left( {V} {\Omega}^{-p_{d,2}} {V}^H \hat{{S}}'_d \right)$
		\State Combine dimensions: $P' = [{f}'_x, {f}'_y, {f}'_z]$
		
	\end{algorithmic}
\end{algorithm}
\subsection{Optimal 3D Point Cloud Filtering in the PMFHT Domain}
\label{sec:5}
Let a point cloud manifold consist of $N$ sample points. Denote by $x\in\mathbb{C}^N$ the vector of values of a scalar field on the point cloud. The observed signal model is shown below:
\begin{equation}
y = c + x+n,\tag{81}
\end{equation}
where $c$ is the clutter, $x$ is the target signal and $n$ is additive noise. According to the idea of optimal Wiener filtering, we consider designing a linear estimation operator to obtain the filtered output signal:
\begin{equation}
\hat x = \mathcal{G} y.\tag{82}
\end{equation}

According to the minimum MSE criterion, the optimal filtering problem on the point cloud manifold can be formulated as:
\begin{align}
	\min_{\mathcal{G}} \mathbb{E}\{ \|\hat x - {x} \|_{{B}}^2 \} &= \mathbb{E}\bigl\{ (\mathcal{G}{y} - {x})^H {B} (\mathcal{G}{y} - {x}) \bigr\} \notag\\
	&= \mathbb{E}\bigl\{ {y}^H \mathcal{G}^H {B} \mathcal{G} {y} - {y}^H \mathcal{G}^H {B} {x} - {x}^H {B} \mathcal{G} {y} + {x}^H {B} {x} \bigr\} \notag\\
	&= \operatorname{Tr}\bigl( \mathcal{G}^H {B} \, \mathcal{G} \, \mathbb{E}\{{y}{y}^H\} \bigr)
	- \operatorname{Tr}\bigl( \mathcal{G}^H {B} \, \mathbb{E}\{{x}{y}^H\} \bigr)
	- \operatorname{Tr}\bigl( \mathbb{E}\{{y}{x}^H\} {B} \mathcal{G} \bigr)
	+ \mathbb{E}\{{x}^H {B} {x}\} \notag\\
	&= \operatorname{Tr}\bigl( \mathcal{G}^H {B} \mathcal{G} {R}_{yy} \bigr)
	- \operatorname{Tr}\bigl( \mathcal{G}^H {B} {R}_{xy} \bigr)
	- \operatorname{Tr}\bigl({R}_{xy}^H {B} \mathcal{G} \bigr)
	+ \text{const}. \tag{83}
	\label{eq:61}
\end{align}
Taking the derivative of the objective with respect to \(\mathcal{G}^H\) and setting it to zero yields \({B}\mathcal{G}{R}_{yy} = {B}{R}_{xy}\). If \({B}\) is invertible, we obtain \(\mathcal{G}{R}_{yy} = {R}_{xy}\), so the optimal linear estimator is
\begin{equation}
\arg \min_{\mathcal{G}} \mathbb{E}\{\|\mathcal{G}{y} - {x}\|_B^2\} = \mathcal{G}_{\text{opt}} = {R}_{xy} {R}_{yy}^{-1} = \mathbb{E}\{{x}{y}^H\}\bigl(\mathbb{E}\{{y}{y}^H\}\bigr)^{-1}. \tag{84}
\end{equation}
We consider estimators of the multiplicative spectral form in the $\alpha$-order fractional domain:
\begin{equation}
\hat x = B^{-1/2}F_M^{- (\alpha)}\, {\Gamma} \, F_M^{(\alpha)}B^{1/2}\, y.\tag{85}
\label{59}
\end{equation}
where ${\Gamma} \in \mathbb{C}^{N \times N}$ representing the multiplicative filter in the fractional manifold harmonic domain.

Based on the optimal fractional Fourier domain filtering framework proposed in \cite{kutay1997optimal}, \cite{9435933}
the optimal filtering theory on the fractional domain of the point cloud manifold is established. 
The processing procedure of this method can be summarized as follows: first, the input point cloud signal is transformed into the manifold fractional harmonic spectral domain; then, a multiplicative filtering operation with specific filtering characteristics is performed in this domain. 
Finally, the processed signal is reconstructed back into the original spatial domain through the inverse fractional harmonic transform. Since the point cloud is discrete, the above multiplicative filtering operation can be implemented by multiplying with a diagonal matrix, which constrains the system matrix ${\Gamma}$ in \eqref{59} to a diagonal form. 
\begin{definition}
To this end, we define the research objective as solving the following optimization problem, 
where the matrix set $D$ is given by
$D \triangleq \{ {\Gamma} \mid {\Gamma} \in \mathbb{C}^{N \times N}, \; H \text{ is a diagonal matrix} \}$.
\end{definition}
Among all possible fractional orders $\alpha$, find the filter matrix H that minimizes the mean square error. The MSE cost is:
\begin{equation}
J(H; \alpha) = \mathbb{E}\{\|B^{-1/2}F_M^{- (\alpha)} {\Gamma} F_M^{(\alpha)} B^{1/2}y - x\|_2^2\},
\tag{86}
\end{equation}
with the optimization problem:
\begin{equation}
\min_{{\Gamma}\ \mathrm{ }}\; J(H; \alpha).\tag{87}
\label{eq:65}
\end{equation}
When $\alpha=1$, this problem reduces to diagonal filtering in the ordinary manifold harmonic domain, which is equivalent to linear shift-invariant (LSI) filtering on the manifold.
Parameter of interest is simply a vector in \( \mathbb{C}^N \) as the number of nonzero elements of \( {\Gamma} \) is at most \( N \). To that end, let \( {F}_M^{(-\alpha)} = [w_1, w_2, \ldots ,w_N] \), \( ({F}_M^{(\alpha)})^T = [\hat{w}_1, \hat{w}_2, \ldots, \hat{w}_N] \), and \( W_i = w_i \hat{w}_i^T \) for \( i = 1, 2, \ldots, N \). Then, for \( {\Gamma} = \text{diag}(h_1, h_2, \ldots, h_N) \), we can write:
\[
{F}_M^{(-\alpha)} {\Gamma} {F}_M^{(\alpha)} = \sum_{i=1}^{N} h_i W_i,\tag{88}
\]
That is, the objective is a function of \( h = [h_1 h_2 \ldots h_N]^T \), and \eqref{eq:65} can be transformed into the following problem:
\[
\min_h \mathbb{E} \left\{ \left\| \sum_{i=1}^{N} h_i W_i y - x \right\|_2^2 \right\},\tag{89}
\]
where we do not impose any restrictions on \( h \).

The optimal filter coefficients are obtained by the following equation:

\[
T h = q,\tag{90}
\]
where
\begin{align*}
T = \mathbb{E}\{S^H S\}, \quad q = \mathbb{E}\{S^H x\}, 
\\
\quad S = [S_1, S_2, \ldots, S_N], \quad S_i = W_{i}y.
\tag{91}
\end{align*}

Using linear system solution:

\[
h^{(\text{opt})} = T^{-1}q.\tag{92}
\]

Construct the filter matrix:
\begin{equation*}
{\Gamma}^{(\text{opt})} = \text{diag}(h^{(\text{opt})}).\tag{93}
\end{equation*}
The filtered output is:
\begin{equation*}
\tilde{x} = {B}^{-1/2}{F}_M^{(-\alpha)} {\Gamma}^{(\text{opt})} {F}_M^{(\alpha)}{B}^{1/2} y.\tag{94}
\end{equation*}

\begin{algorithm}[htbp]
\caption{Optimal Fractional-Domain Filtering for Clutter Suppression}\label{alg:optimal_filtering}
\begin{algorithmic}[1]
\State \textbf{Input:} Radar echo point cloud data $y$, fractional order $\alpha$
\State \textbf{Output:} Filtered target signal $\tilde{x}$
\State Compute $F_M^{(\alpha)}$ and $F_M^{(-\alpha)}$ using Algorithm \ref{alg:PMFHT_normalized_orthogonal} 
\State ${F}_M^{(-\alpha)} = [w_1, w_2\ldots w_N]$, $({F}_M^{(\alpha)})^T = [\hat{w}_1, \hat{w}_2\ldots\hat{w}_N]$, 
construct basis matrices: $W_i = w_i \hat{w}_i^T$ for $i = 1, \ldots, N$
\State \text{Compute cross-statistics matrices:}
 $\mathbf{S} = [S_1, \ldots, S_N]$ ,
where $S_i = W_i y$
\State \text{Build linear system:}
\hspace{0em} $T = \mathbb{E}\{S^H S\},$
\hspace{0em} $q = \mathbb{E}\{S^H x\}$
\State \text{Solve for optimal filter coefficients:}
\hspace{0em} $h^{(opt)} = T^{-1} q$, 
${\Gamma}^{(opt)} = \operatorname{diag}(h^{(opt)})$
\State \text{Apply optimal filtering:}
$\tilde{x} = {B}^{-1/2}F_M^{(-\alpha)} {\Gamma}^{(opt)} F_M^{(\alpha)}{B}^{1/2} y$
\end{algorithmic}
\end{algorithm}

\section{Simulation Experiments}
\label{sec:6}
\subsection{Point Cloud Data Encryption}
\label{sec:exp_encryption}

To evaluate the effectiveness and security of the proposed encryption algorithm, experiments are conducted on multiple 3D point cloud datasets, including both simulated data and publicly available models. The public datasets are sourced from the Princeton Shape Benchmark (\url{https://gfx.cs.princeton.edu/proj/sugcon/models/}) and the Stanford 3D Scanning Repository (\url{https://graphics.stanford.edu/data/3Dscanrep/}). Representative models such as the \textit{Swiss roll}, \textit{Bunny}, and \textit{Lion} are employed to demonstrate the performance of the proposed algorithms across diverse geometric topologies.

The encryption process follows the architecture described in Section~\ref{sec:4}. The secret key $K$ used for the simulation is configured as follows: the forward fractional orders are set to $\alpha_{fwd} = [0.35, 0.72, 0.15]$, and the secondary inverse orders are $\alpha_{inv} = [0.60, 0.20, 0.90]$. The chaotic system parameters for the Hénon map are chosen as $(a=1.4, b=0.3)$ with initial values $(u_0=0.12, v_0=0.1)$. These parameters collectively form an extremely large key space, providing robust resistance against brute-force attacks.

The experimental results are illustrated in Fig.~\ref{fig:encryption_results_main}. As observed in the second column of the figure, the original geometric structures of the point clouds are completely masked after encryption, appearing as isotropic, noise-like distributions in the 3D space. This indicates that the combination of multi-order PMFHT and chaotic spectral diffusion effectively disrupts the spatial correlation of the coordinates.

To verify the sensitivity of the algorithm to the fractional orders, a decryption test with an incorrect key is performed. While the correct decryption (fourth column) perfectly restores the original geometry with a relative error of approximately $10^{-12}$, the error-decryption attempt (third column) utilizes a modified inverse order set $\alpha_{inv}' = [0.10, 0.20, 0.90]$. Despite the deviation occurring in only a single parameter ($p_{x,2}$ changed from $0.60$ to $0.10$), the recovered data remains a disordered cloud of points with no discernible features. This high sensitivity to fractional orders, coupled with the non-linear properties of the chaotic phase masks, ensures that the proposed scheme provides a high level of privacy protection for 3D point cloud data during transmission and storage.
\begin{figure}[htbp] 
	\centering
	
	\begin{minipage}{0.7\textwidth}
		\centering
		\includegraphics[width=\linewidth]{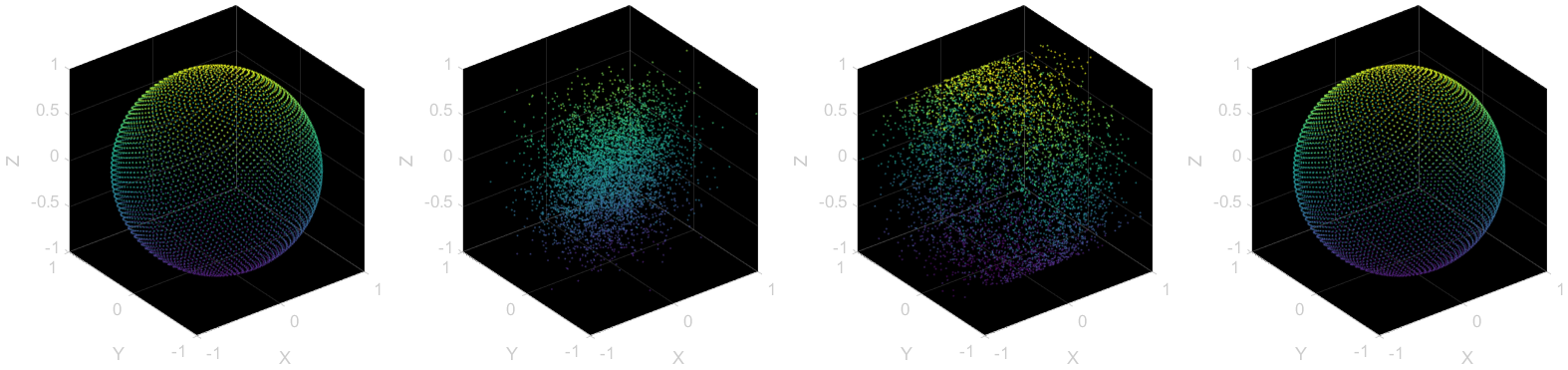}
	\end{minipage}
	\begin{minipage}{0.04\textwidth}
		\footnotesize (a)
	\end{minipage}
	\vspace{1mm}
	
	\begin{minipage}{0.7\textwidth}
		\centering
		\includegraphics[width=\linewidth]{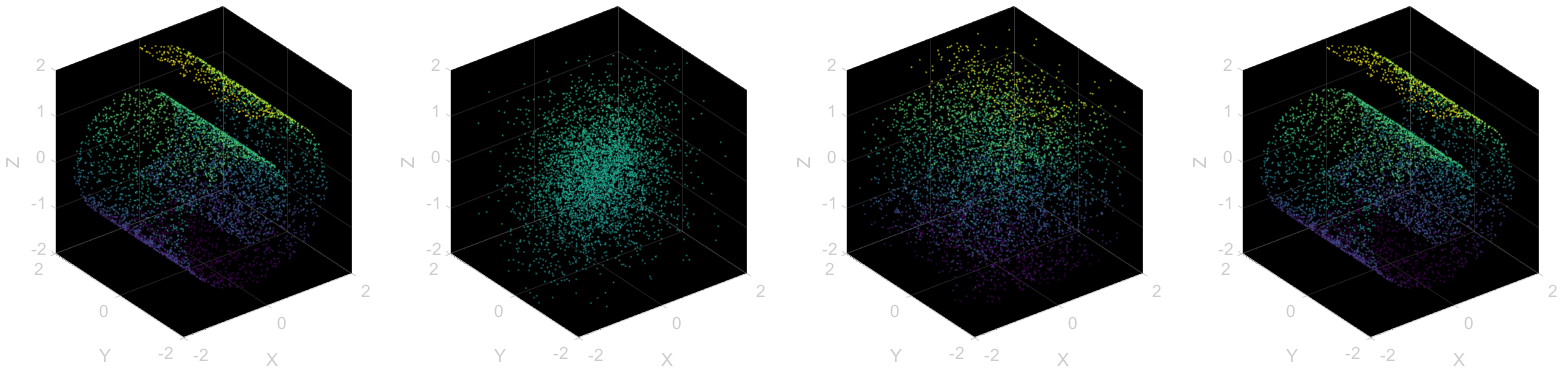}
	\end{minipage}
	\begin{minipage}{0.04\textwidth}
		\footnotesize (b)
	\end{minipage}
	\vspace{1mm}
	
	\begin{minipage}{0.7\textwidth}
		\centering
		\includegraphics[width=\linewidth]{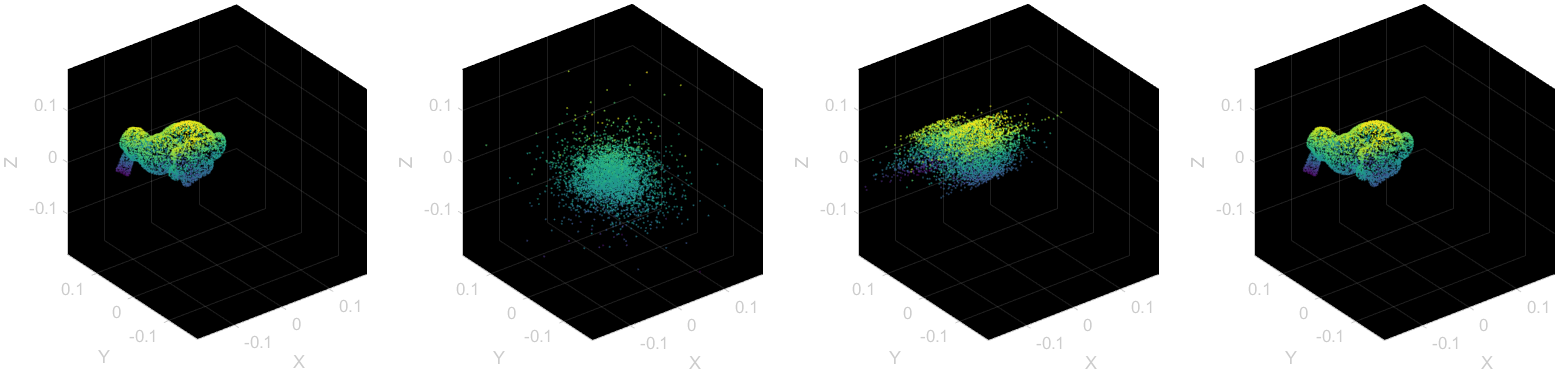}
	\end{minipage}
	\begin{minipage}{0.04\textwidth}
		\footnotesize (c)
	\end{minipage}
	\vspace{1mm}
	
	\begin{minipage}{0.7\textwidth}
		\centering
		\includegraphics[width=\linewidth]{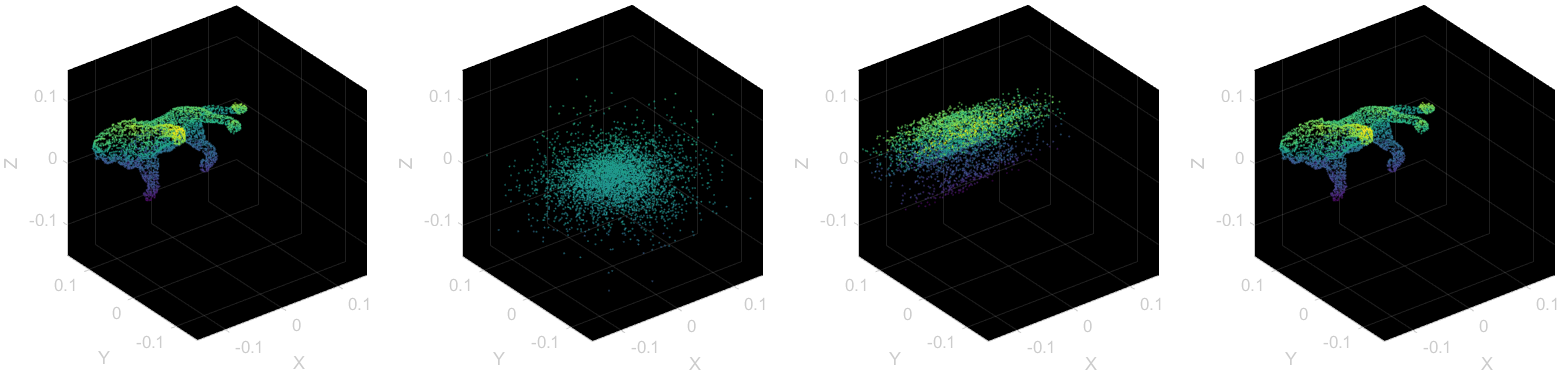}
	\end{minipage}
	\begin{minipage}{0.04\textwidth}
		\footnotesize (d)
	\end{minipage}
	\vspace{1mm}
	
	\begin{minipage}{0.7\textwidth}
		\centering
		\includegraphics[width=\linewidth]{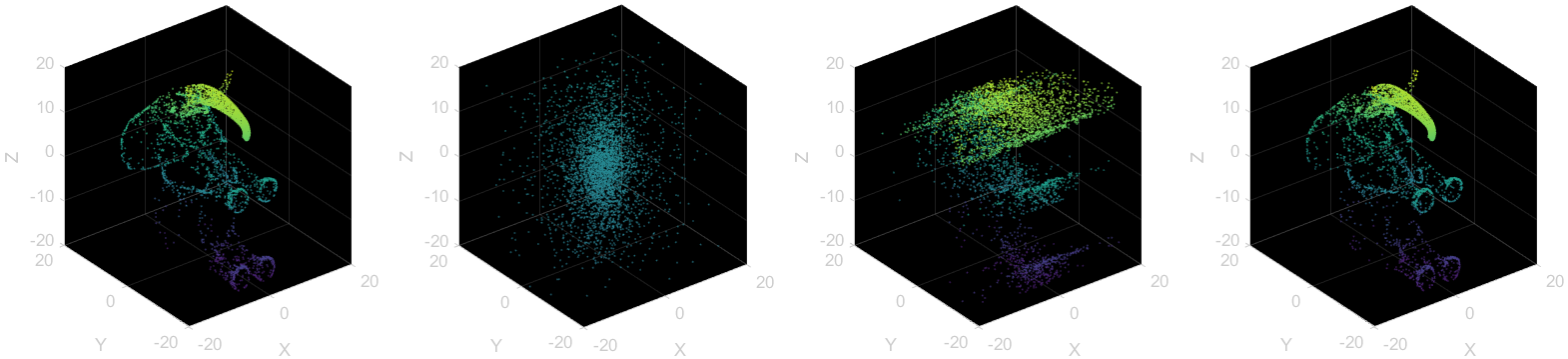}
	\end{minipage}
	\begin{minipage}{0.04\textwidth}
		\footnotesize (e)
	\end{minipage}
	\vspace{1mm}
	
	\begin{minipage}{0.7\textwidth}
		\centering
		\includegraphics[width=\linewidth]{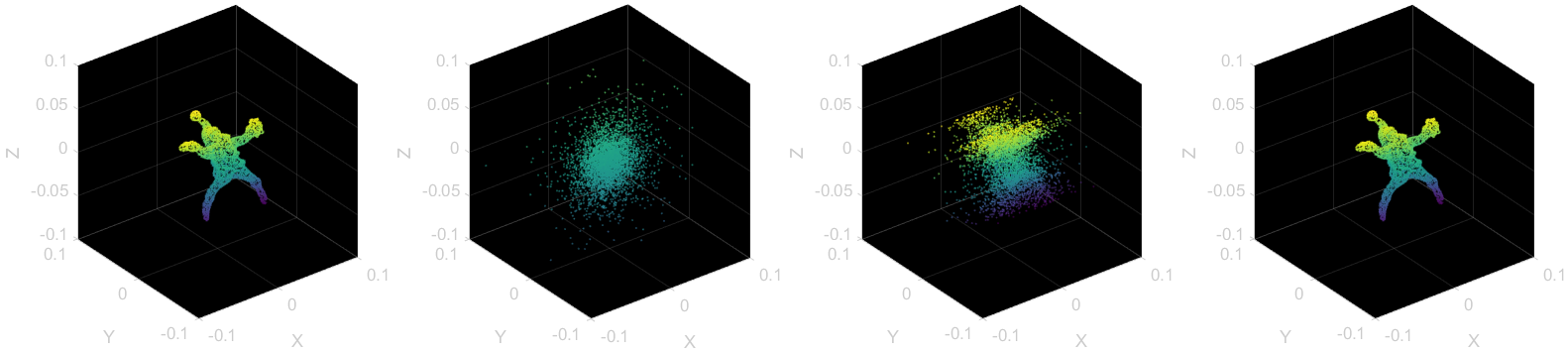}
	\end{minipage}
	\begin{minipage}{0.04\textwidth}
		\footnotesize (f)
	\end{minipage}
	
	\vspace{2mm} 
	
	\begin{minipage}{0.7\textwidth}
		\centering\footnotesize 
		\begin{minipage}{0.24\linewidth}\centering Original\\point cloud\end{minipage}\hfill
		\begin{minipage}{0.24\linewidth}\centering Encrypted\\point cloud\end{minipage}\hfill
		\begin{minipage}{0.24\linewidth}\centering Error-decrypted\\point cloud\end{minipage}\hfill
		\begin{minipage}{0.24\linewidth}\centering Decrypted\\point cloud\end{minipage}
	\end{minipage}
	\hspace{0.04\textwidth} 
	
	\caption{Experimental results of the proposed PMFHT-domain encryption for various 3D models: (a) Sphere, (b) Swiss roll, (c) Bunny, (d) Lion, (e) Elephant, (f) Santa}
	\label{fig:encryption_results_main}
\end{figure}
\subsection{IPIX Sea Clutter Suppression}
To validate the effectiveness of the proposed PMFHT framework for maritime target detection, we conducted experiments on the IPIX radar dataset. 
A synthetic floating target, moving with a constant Doppler frequency corresponding to a radial velocity of 2.58~m/s, was injected into the fourth range cell of the clutter data to simulate a realistic detection scenario. 
For these experiments, we used the dataset file \text{19980205\_170935\_ANTSTEP.CDF} from the McMaster IPIX radar database~\cite{ipix}, which contains pure sea clutter returns without any real target.
The target echo is modeled as $s_t = A {p}$, where $A$ represents the target signal amplitude and ${p}$ is the Doppler steering vector defined by 
${p} = \frac{1}{\sqrt{M}}[1, \exp(-j2\pi f_d/f_r), \dots, \exp(-j2\pi (M-1) f_d/f_r)]^T$, 
with $j$ denoting the imaginary unit, $f_r$ the pulse repetition frequency, $M$ the number of transmitted pulses, and $f_d = 2v/\lambda$ the Doppler frequency corresponding to a radial velocity $v$ and wavelength $\lambda$. Detailed parameter settings are shown in Table~\ref{table1}.
\begin{figure}[htbp]
	\centering
	\subfloat[(a)\label{fig:sub2_orimesh}]{%
		\includegraphics[width=0.32\textwidth]{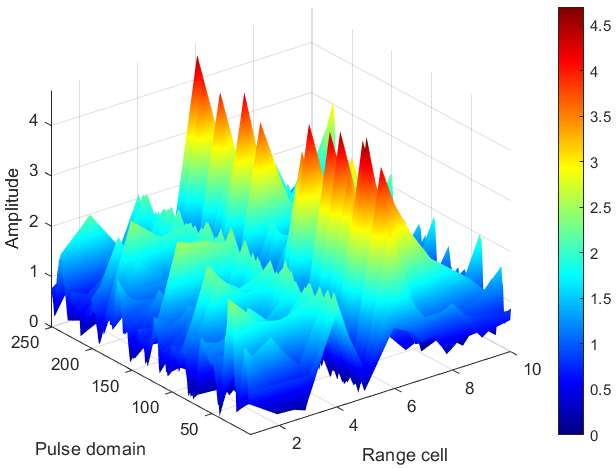}%
	}
	\hspace{1cm}
	\subfloat[(b)\label{fig:sub3_orimesh}]{%
		\includegraphics[width=0.32\textwidth]{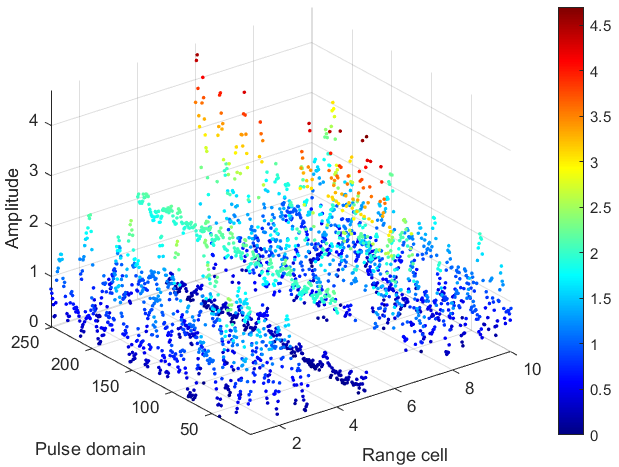}%
	}
	\caption{Visualization of the raw IPIX data in the range–pulse domain. (a) Mesh representation. (b) Point cloud representation.}
	\label{fig:7}
\end{figure}
\begin{figure}[htbp]
	\centering
	
	\subfloat[(a)\label{fig:sub2_filteredmesh}]{%
		\includegraphics[width=0.32\textwidth]{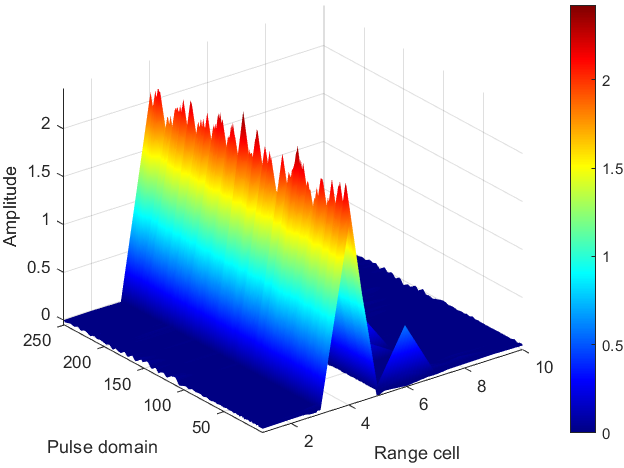}%
	}
	\hspace{1cm} 
	\subfloat[(b)\label{fig:sub3_filteredmesh}]{%
		\includegraphics[width=0.32\textwidth]{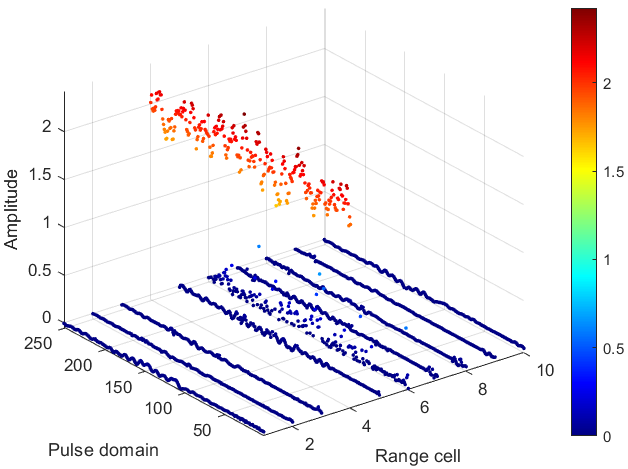}%
	}
	
	\caption{Target detection results using PMFHT optimal filtering. (a) Mesh representation. (b) Point cloud representation.}
	\label{fig:8}
\end{figure}

Fig.~\ref{fig:7} illustrates the raw IPIX data visualization in the range-pulse domain. The front-view surface plot in Fig.~\ref{fig:7}(a) clearly depicts the spatio-temporal characteristics of sea clutter, while Fig.~\ref{fig:7}(b) shows the 3D point cloud representation used as input to our PMFHT framework.
The experimental procedure is as follows. First, the raw radar echo data were preprocessed to form a 3D point cloud, where each point corresponds to a range-pulse-power sample. The point cloud was then modeled as a discrete manifold, and the LBO was constructed following the methodology described in Section~\ref{sec:2}. The PMHB and PMFHT were computed for different fractional orders $\alpha$. Fig.~\ref{fig:spectral_of_xyz} displays the fractional harmonic spectra of the unfiltered data across the x, y, and z spectral dimensions, revealing the complex geometric structure of sea clutter in the fractional domain.
\begin{table}[!t]
	\renewcommand{\arraystretch}{1.0} 
	\centering
	\caption{ \\[1mm] IPIX Radar System and Experimental Configuration}
	\centering
	\begin{tabular}{lc}
		\toprule
		\textbf{Parameter} & \textbf{Value} \\
		\midrule
		Carrier Frequency & 9.39 GHz \\
		Wavelength & 0.03 m \\
		Peak Power & 8 kW \\
		Antenna Type & Dual-Polarized Parabolic \\
		Antenna Size & 2.4 m \\
		Antenna Gain & 45.7 dB \\
		Beamwidth & 1.1$^\circ$ \\
		Polarization & Dual Linear (full matrix) \\
		Receiver Channels & 2 \\
		Number of Range Cells Used & 10 \\
		Number of Pulses Used & 251 \\
		Target Range Cell & 4 \\
		Target Velocity & 2.58 m/s \\
		Pulse Repetition Frequency & 1075 Hz \\
		\bottomrule
	\end{tabular}
	\label{table1}
\end{table}
\begin{figure}[htbp]
	\centering
	
	\subfloat[(a)\label{fig:sub1_origin}]{%
		\includegraphics[width=0.31\columnwidth]{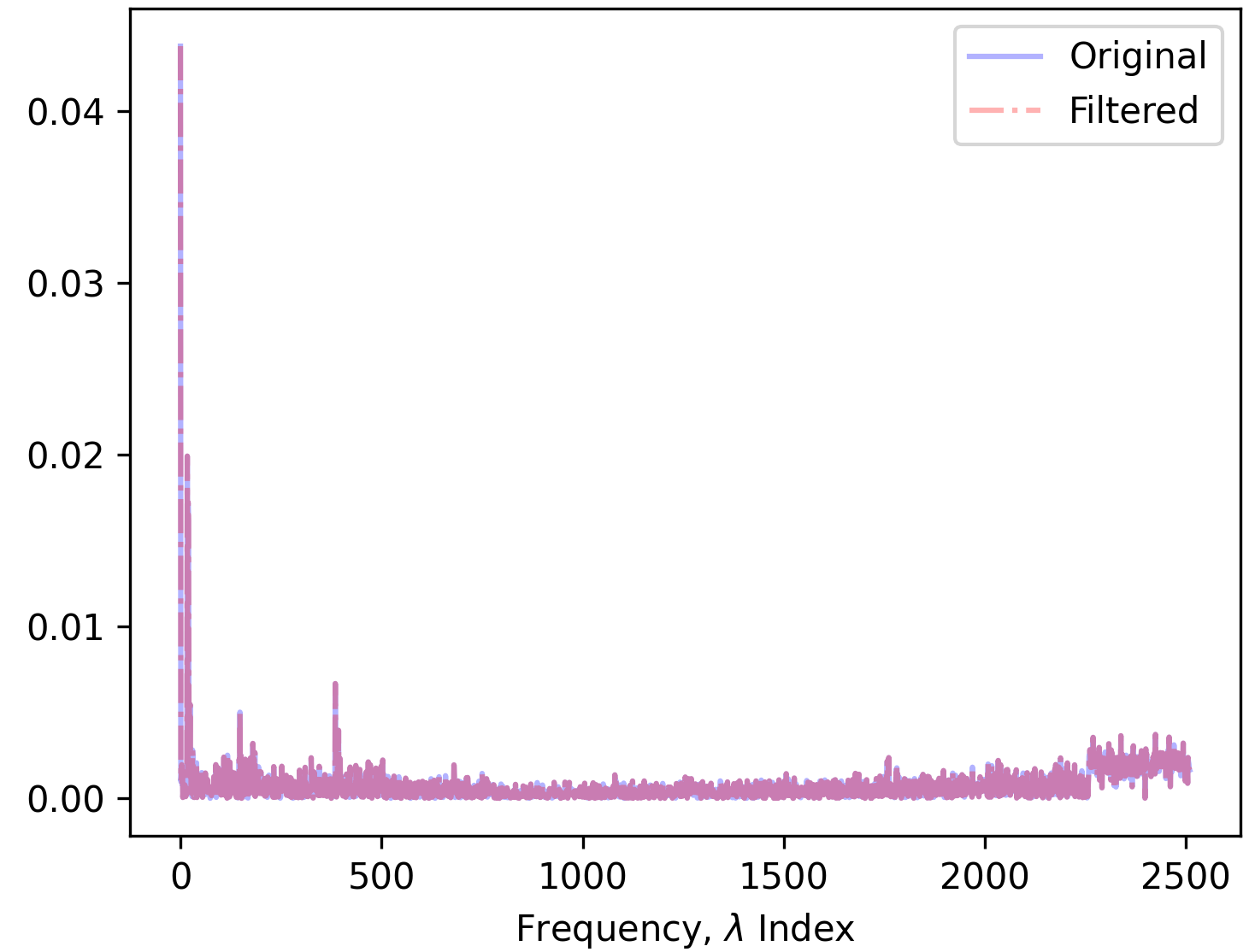}%
	}
	\hfill
	\subfloat[(b)\label{fig:sub2_origin}]{%
		\includegraphics[width=0.31\columnwidth]{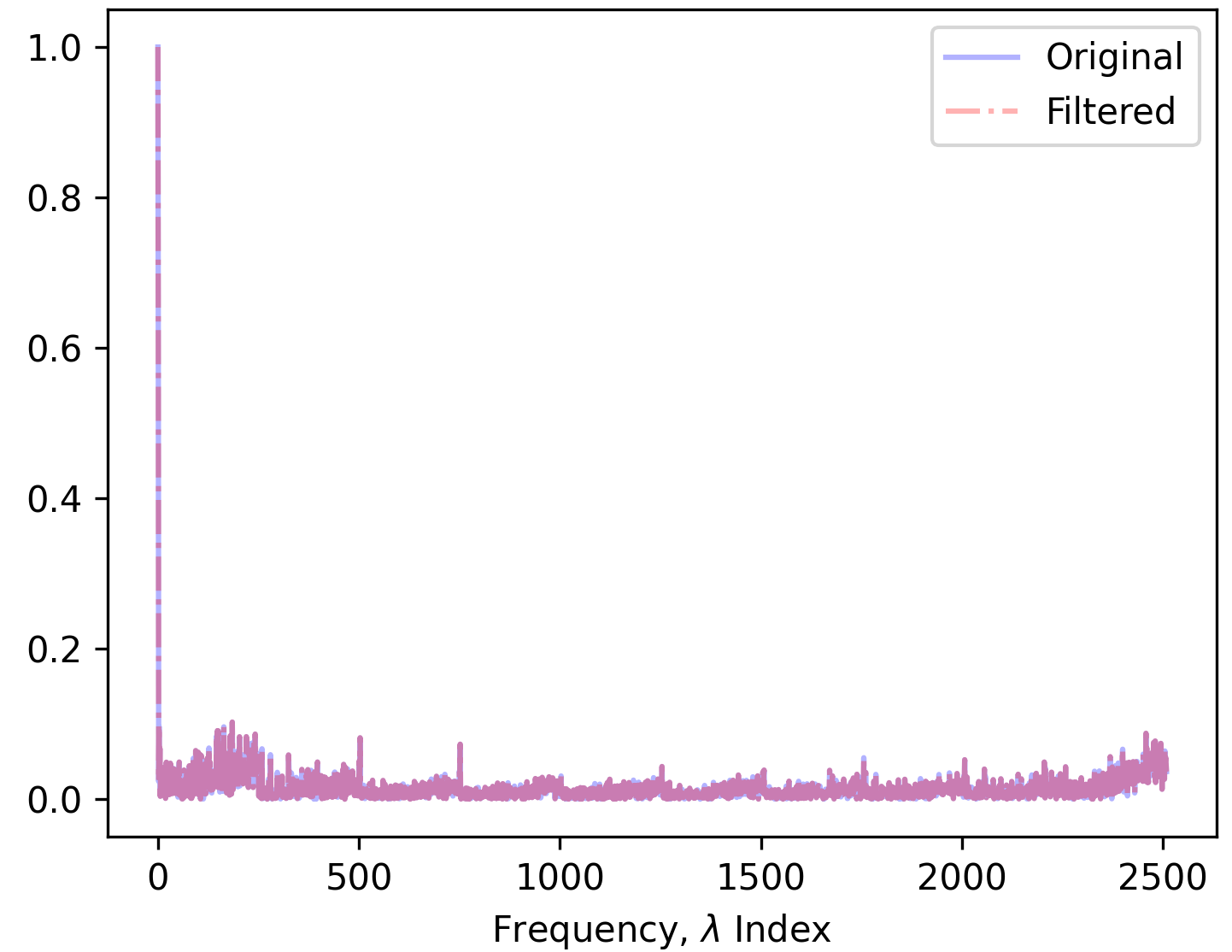}%
	}
	\hfill
	\subfloat[(c)\label{fig:sub3_origin}]{%
		\includegraphics[width=0.31\columnwidth]{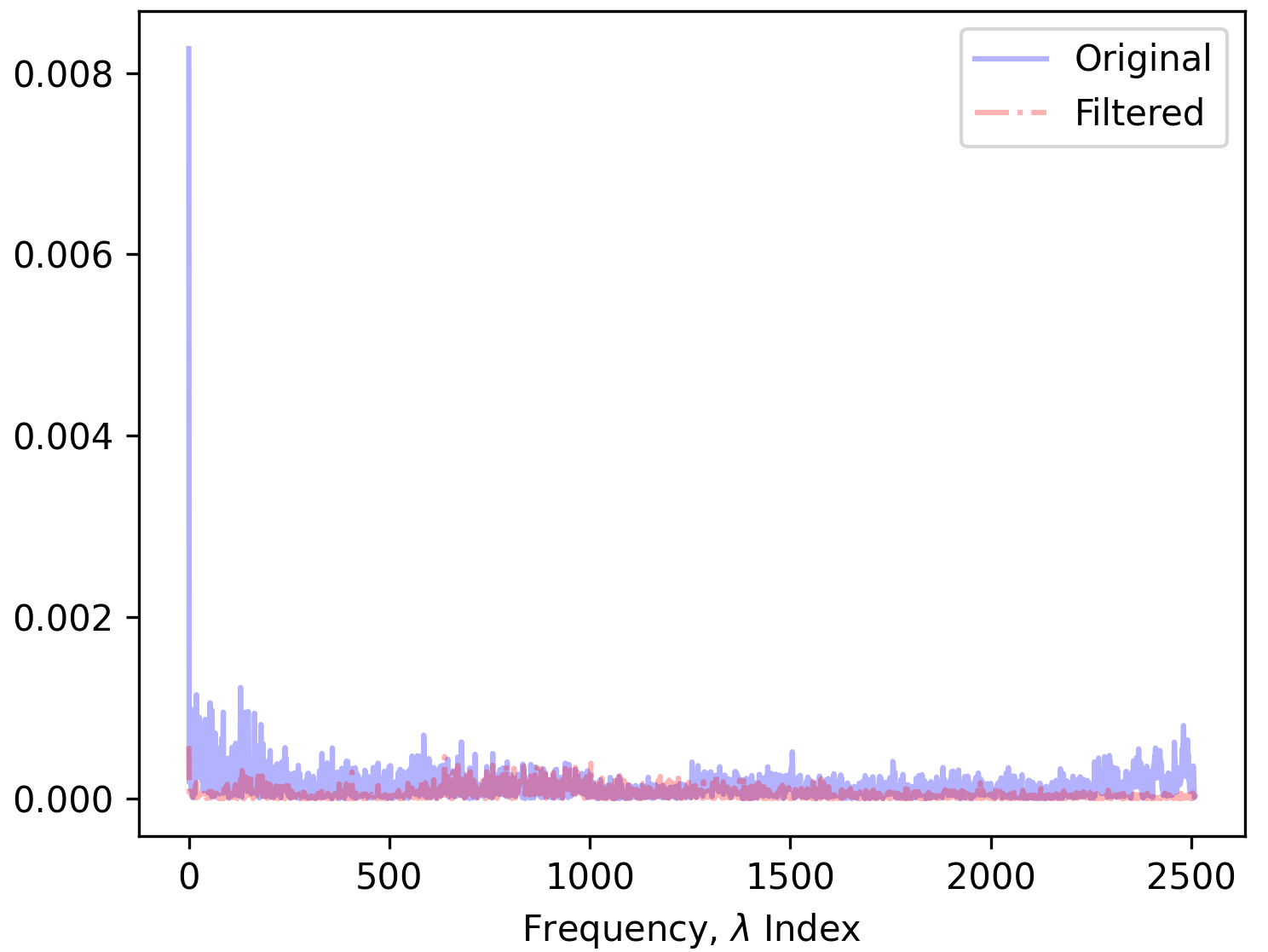}%
	}
	
	\caption{PMFHT spectra of IPIX point cloud data. (a) x-spectrum. (b) y-spectrum. (c) z-spectrum.}
	\label{fig:spectral_of_xyz}
\end{figure}
\begin{figure}[htbp]
	\centering
	\subfloat[(a)\label{fig:alpha_mse}]{%
		\includegraphics[width=0.31\columnwidth]{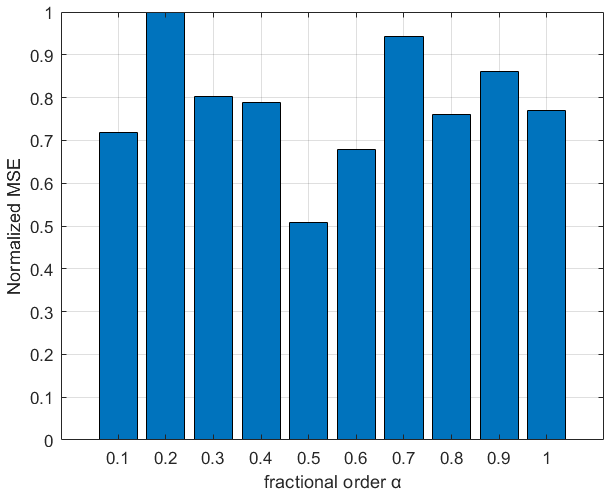}%
	}
	\hspace{1cm}
	\subfloat[(b)\label{fig:scr_pd}]{%
		\includegraphics[width=0.31\columnwidth]{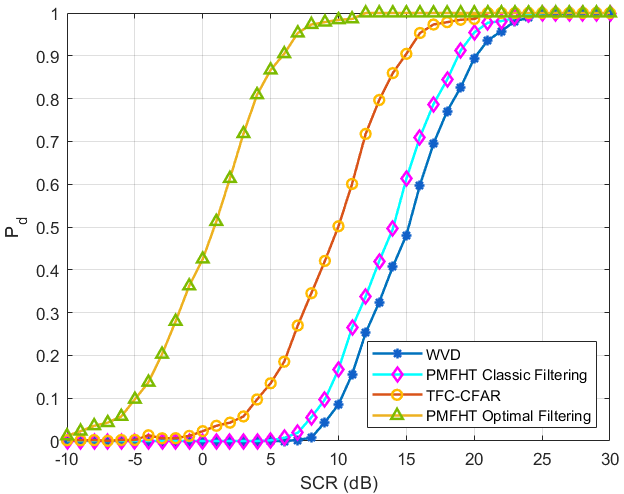}%
	}
	\caption{(a) Normalized MSE for different values of $\alpha$. (b) Detection performance of different methods on measured data.}
	\label{fig:11}
\end{figure}
Next, the optimal fractional domain filtering method derived in Section~\ref{sec:5} was applied to suppress sea clutter and achieve target detection. In constructing the optimal fractional domain filter, this paper employs a sliding window local estimator. For each range cell, averaging the pulse-dimension samples of that cell and its two neighbors provides a compact estimate of the target signal components required for the optimal PMFHT filter. 
Fig.~\ref{fig:8} presents the target detection results obtained using the PMFHT optimal filter, where Fig.~\ref{fig:8}(a) shows significant clutter suppression in the front view, and Fig.~\ref{fig:8}(b) illustrates the enhanced target visibility in the filtered point cloud. The corresponding fractional harmonic spectra of the PMFHT-filtered data are shown in Fig.~\ref{fig:spectral_of_xyz}, where clutter components are effectively attenuated while preserving target information across all spectral dimensions. It can be observed that the z-component spectrum exhibits a significant change after filtering, whereas the x-component and y-component spectra remain largely unchanged. This is because the z-component reflects the energy dimension of the signal, and the objective of the proposed method is essentially to suppress the sea-clutter energy while preserving the target energy. Therefore, the filtering applied to the z-component spectrum plays a dominant role in the overall target detection performance.

Fig.~\ref{fig:11}(a) shows the normalized mean square error (NMSE) of the reconstructed target signal under different fractional orders $\alpha$. It can be seen that the MSE is minimized when the fractional order $\alpha = 0.5$, indicating that PMFHT provides a more discriminative spectral representation compared to the traditional PMHT ($\alpha = 1$), and can better separate the target from clutter. Performance is evaluated in terms of detection probability ($P_d$) versus SCR curves, with results averaged over 200 Monte Carlo trials to ensure statistical reliability. Fig.~\ref{fig:11}(b) compares the detection performance of the proposed optimal PMFHT-based filter with several baseline methods, including the traditional PMFHT bandpass filter, the Wigner-Ville distribution (WVD) detection method, and the time-frequency correlation CFAR detection method \cite{cao2021nonstationary}. The results show that the PMFHT detector achieves a higher detection probability, especially in low SCR scenarios. This confirms that fractional tuning in PMFHT can achieve flexible and effective clutter suppression, especially suitable for non-stationary sea spikes.

In summary, the application of PMFHT to maritime target detection demonstrates its ability to suppress sea clutter and enhance target visibility in complex sea clutter environments. By utilizing the fractional spectral domain, this method achieves excellent performance in real-world data scenarios, providing a robust and efficient solution for radar-based maritime surveillance.

\section{Conclusion}
\label{sec:7}
Based on the theory of Riemannian manifolds, this paper proposes a fractional harmonic transform for 3D point cloud data, extending the spectral representation of point clouds to multi-order fractional domains and establishing a complete theoretical framework, including definitions, properties, convolution, correlation, and sampling. Combined with the PMFHT, two algorithms are developed for point cloud encryption and point cloud clutter suppression, whose effectiveness is verified by real-world experimental data. In summary, the proposed PMFHT provides a novel and powerful tool for point cloud processing and exhibits promising application potential in fields such as neural networks, autonomous driving, LiDAR, and data security.

\section*{A\lowercase{uthor} C\lowercase{ontributions}}
\textbf{Jiamian Li}: methodology; software; visualization; writing - original draft. \textbf{Bing-Zhao Li}: funding acquisition; methodology; resources; writing - review \& editing.

\section*{A\lowercase{cknowledgements}}
This work was supported by grants from Natural Science Foundation of Beijing Municipality [No. 4242011].

\section*{C\lowercase{onflicts} \lowercase{of} I\lowercase{nterest}}
This work does not have any conflicts of interest.

\section*{D\lowercase{ata} A\lowercase{vailability} S\lowercase{tatement}}
Data available on request from the authors.

\bibliographystyle{wileyNJD-Chicago}
\bibliography{wileyNJD-Chicago}

\begin{thebibliography}{31}
\providecommand{\natexlab}[1]{#1}
\providecommand{\url}[1]{\texttt{#1}}
\expandafter\ifx\csname urlstyle\endcsname\relax
  \providecommand{\doi}[1]{doi: #1}\else
  \providecommand{\doi}{doi: \begingroup \urlstyle{rm}\Url}\fi

\bibitem[\protect\citeauthoryear{Han, Jin, Wang, Jiang, Gao, and Xiao}{Han
  et~al.}{}]{han2017review}
X.-F. Han, J.~S. Jin, M.-J. Wang, W.~Jiang, L.~Gao, and L.~Xiao.
\newblock
\newblock ``A review of algorithms for filtering the 3D point cloud.''  {\it
  Signal Processing: Image Communication\/}57 (2017): 103--112.

\bibitem[\protect\citeauthoryear{Zhang, Wang, Tian, Lu, Zhang, Ning, and
  Bai}{Zhang et~al.}{}]{zhang2023deep}
H.~Zhang, C.~Wang, S.~Tian, B.~Lu, L.~Zhang, X.~Ning, and X.~Bai.
\newblock
\newblock ``Deep learning-based 3D point cloud classification: A systematic
  survey and outlook.''  {\it Displays\/}79 (2023): 102456.

\bibitem[\protect\citeauthoryear{Guo, Wang, Hu, Liu, Liu, and Bennamoun}{Guo
  et~al.}{}]{guo2020deep}
Y.~Guo, H.~Wang, Q.~Hu, H.~Liu, L.~Liu, and M.~Bennamoun.
\newblock
\newblock ``Deep learning for 3d point clouds: A survey.''  {\it IEEE
  Transactions on Pattern Analysis and Machine Intelligence\/}~43, no. 12
  (2020): 4338--4364.

\bibitem[\protect\citeauthoryear{Lee}{Lee}{}]{lee2003smooth}
J.~M. Lee 2003.
\newblock ``Smooth manifolds.''   In {\it Introduction to smooth manifolds},
  1--29.
\newblock Springer.

\bibitem[\protect\citeauthoryear{Martin and Watson}{Martin and
  Watson}{}]{MARTIN2011427}
S.~Martin and J.-P. Watson.
\newblock
\newblock ``Non-manifold surface reconstruction from high-dimensional point
  cloud data.''  {\it Computational Geometry\/}~44, no. 8 (2011): 427--441.

\bibitem[\protect\citeauthoryear{Lee, Kim, Choi, and Park}{Lee
  et~al.}{}]{lee2022statistical}
Y.~Lee, S.~Kim, J.~Choi, and F.~Park. 2022.
\newblock ``A statistical manifold framework for point cloud data.''   In {\it
  International Conference on Machine Learning},   12378--12402.
\newblock PMLR.

\bibitem[\protect\citeauthoryear{Pauly and Gross}{Pauly and
  Gross}{}]{pauly2001spectral}
M.~Pauly and M.~Gross. 2001.
\newblock ``Spectral processing of point-sampled geometry.''   In {\it
  Proceedings of the 28th annual conference on Computer graphics and
  interactive techniques},   379--386.

\bibitem[\protect\citeauthoryear{Vallet and L{\'e}vy}{Vallet and
  L{\'e}vy}{}]{vallet2008spectral}
B.~Vallet and B.~L{\'e}vy. 2008.
\newblock ``Spectral geometry processing with manifold harmonics.''   In {\it
  Computer Graphics Forum}, Vol~27,   251--260.
\newblock Wiley Online Library.

\bibitem[\protect\citeauthoryear{Liu, Prabhakaran, and Guo}{Liu
  et~al.}{}]{liu2012point}
Y.~Liu, B.~Prabhakaran, and X.~Guo.
\newblock
\newblock ``Point-based manifold harmonics.''  {\it IEEE Transactions on
  visualization and computer graphics\/}~18, no. 10 (2012): 1693--1703.

\bibitem[\protect\citeauthoryear{Bultheel and Mart{\'i}nez~Sulbaran}{Bultheel
  and Mart{\'i}nez~Sulbaran}{}]{BULTHEEL2004182}
A.~Bultheel and H.~E. Mart{\'i}nez~Sulbaran.
\newblock
\newblock ``Computation of the fractional Fourier transform.''  {\it Applied
  and Computational Harmonic Analysis\/}~16, no. 3 (2004): 182--202.

\bibitem[\protect\citeauthoryear{Chen, Fu, Grafakos, and Wu}{Chen
  et~al.}{}]{CHEN202171}
W.~Chen, Z.~Fu, L.~Grafakos, and Y.~Wu.
\newblock
\newblock ``Fractional Fourier transforms on Lp and applications.''  {\it
  Applied and Computational Harmonic Analysis\/}55 (2021): 71--96.

\bibitem[\protect\citeauthoryear{Fu, Grafakos, Lin, Wu, and Yang}{Fu
  et~al.}{}]{FU2023211}
Z.~Fu, L.~Grafakos, Y.~Lin, Y.~Wu, and S.~Yang.
\newblock
\newblock ``Riesz transform associated with the fractional Fourier transform
  and applications in image edge detection.''  {\it Applied and Computational
  Harmonic Analysis\/}66 (2023): 211--235.

\bibitem[\protect\citeauthoryear{Tao, Meng, and Wang}{Tao et~al.}{}]{5551197}
R.~Tao, X.~Meng, and Y.~Wang.
\newblock
\newblock ``Image Encryption With Multiorders of Fractional Fourier
  Transforms.''  {\it IEEE Transactions on Information Forensics and
  Security\/}~5, no. 4 (2010): 734--738.

\bibitem[\protect\citeauthoryear{Farah, Guesmi, Kachouri, and Samet}{Farah
  et~al.}{}]{FARAH2020105777}
M.A.~Ben Farah, R.~Guesmi, A.~Kachouri, and M.~Samet.
\newblock
\newblock ``A novel chaos based optical image encryption using fractional
  Fourier transform and DNA sequence operation.''  {\it Optics $\&$ Laser
  Technology\/}121 (2020): 105777.

\bibitem[\protect\citeauthoryear{Liu, Liu, Xie, Jiang, Ye, Song, Chai, Liu,
  Feng, and Yuan}{Liu et~al.}{}]{Liu:23}
B.~Liu, Y.~Liu, Y.~Xie, X.~Jiang, Y.~Ye, T.~Song, J.~Chai, M.~Liu, M.~Feng, and
  H.~Yuan.
\newblock
\newblock ``Privacy protection for 3D point cloud classification based on an
  optical chaotic encryption scheme.''  {\it Optics Express\/}~31, no. 5
  (2023): 8820--8843.

\bibitem[\protect\citeauthoryear{Ru, Xu, He, and Shui}{Ru
  et~al.}{}]{ru2023marine}
H.~Ru, S.~Xu, Q.~He, and P.~Shui.
\newblock
\newblock ``Marine small floating target detection method based on fusion
  weight and graph dynamic attention mechanism.''  {\it IEEE Transactions on
  Geoscience and Remote Sensing\/}61 (2023): 1--11.

\bibitem[\protect\citeauthoryear{Su, Chen, Guan, Huang, Wang, and Xue}{Su
  et~al.}{}]{su2024radar}
N.~Su, X.~Chen, J.~Guan, Y.~Huang, X.~Wang, and Y.~Xue.
\newblock
\newblock ``Radar maritime target detection via spatial--temporal feature
  attention graph convolutional network.''  {\it IEEE Transactions on
  Geoscience and Remote Sensing\/}62 (2024): 1--15.

\bibitem[\protect\citeauthoryear{Hua, Ono, Peng, Cheng, and Wang}{Hua
  et~al.}{}]{hua2021target}
X.~Hua, Y.~Ono, L.~Peng, Y.~Cheng, and H.~Wang.
\newblock
\newblock ``Target detection within nonhomogeneous clutter via total Bregman
  divergence-based matrix information geometry detectors.''  {\it IEEE
  Transactions on Signal Processing\/}69 (2021): 4326--4340.

\bibitem[\protect\citeauthoryear{Cao, Cheng, Wu, and Wang}{Cao
  et~al.}{}]{cao2021nonstationary}
X.~Cao, Y.~Cheng, H.~Wu, and H.~Wang.
\newblock
\newblock ``Nonstationary moving target detection in spiky sea clutter via
  time-frequency manifold.''  {\it IEEE Geoscience and Remote Sensing
  Letters\/}19 (2021): 1--5.

\bibitem[\protect\citeauthoryear{Li, Chen, and Li}{Li et~al.}{}]{li2025novel}
J.-M. Li, J.-Y. Chen, and B.-Z. Li.
\newblock
\newblock ``A novel STAP algorithm via volume cross-correlation function on the
  Grassmann manifold.''  {\it Digital Signal Processing\/}162 (2025): 105164.

\bibitem[\protect\citeauthoryear{Chen, Guan, Bao, and He}{Chen
  et~al.}{}]{chen2013detection}
X.~Chen, J.~Guan, Z.~Bao, and Y.~He.
\newblock
\newblock ``Detection and extraction of target with micromotion in spiky sea
  clutter via short-time fractional Fourier transform.''  {\it IEEE
  Transactions on Geoscience and Remote Sensing\/}~52, no. 2 (2013):
  1002--1018.

\bibitem[\protect\citeauthoryear{Gao, Tao, and Kang}{Gao
  et~al.}{}]{gao2021weak}
C.~Gao, R.~Tao, and X.~Kang.
\newblock
\newblock ``Weak target detection in the presence of sea clutter using
  radon-fractional Fourier transform canceller.''  {\it IEEE Journal of
  Selected Topics in Applied Earth Observations and Remote Sensing\/}14 (2021):
  5818--5830.

\bibitem[\protect\citeauthoryear{Belkin and Niyogi}{Belkin and
  Niyogi}{}]{BELKIN20081289}
M.~Belkin and P.~Niyogi.
\newblock
\newblock ``Towards a theoretical foundation for Laplacian-based manifold
  methods.''  {\it Journal of Computer and System Sciences\/}~74, no. 8 (2008):
  1289--1308.

\bibitem[\protect\citeauthoryear{Belkin, Sun, and Wang}{Belkin
  et~al.}{}]{Belkin2009}
M.~Belkin, J.~Sun, and Y.~Wang. 2009.
\newblock ``Constructing the Laplace operator from point clouds in
  $\mathbb{R}^d$.''   In {\it Proceedings of the Twentieth Annual ACM-SIAM
  Symposium on Discrete Algorithms},   1031--1040.
\newblock Society for Industrial and Applied Mathematics.

\bibitem[\protect\citeauthoryear{Kamalakkannan and Roopkumar}{Kamalakkannan and
  Roopkumar}{}]{Kamalakkannan}
R.~Kamalakkannan and R.~Roopkumar.
\newblock
\newblock ``Multidimensional fractional Fourier transform and generalized
  fractional convolution.''  {\it Integral Transforms and Special
  Functions\/}~31, no. 2 (2020): 152--165.

\bibitem[\protect\citeauthoryear{Su, Tao, and Kang}{Su
  et~al.}{}]{su2019analysis}
X.~Su, R.~Tao, and X.~Kang.
\newblock
\newblock ``Analysis and comparison of discrete fractional Fourier
  transforms.''  {\it Signal Processing\/}160 (2019): 284--298.

\bibitem[\protect\citeauthoryear{Alika{\c{s}}ifo{\u{g}}lu, Kartal, and
  Ko{\c{c}}}{Alika{\c{s}}ifo{\u{g}}lu et~al.}{}]{alikacsifouglu2024graph}
T.~Alika{\c{s}}ifo{\u{g}}lu, B.~Kartal, and A.~Ko{\c{c}}.
\newblock
\newblock ``Graph Fractional Fourier Transform: A Unified Theory.''  {\it IEEE
  Transactions on Signal Processing\/}72 (2024): 3834--3850.

\bibitem[\protect\citeauthoryear{Wang, Li, and Cheng}{Wang
  et~al.}{}]{wang2017fractional}
Y.-Q. Wang, B.-Z. Li, and Q.-Y. Cheng. 2017.
\newblock ``The fractional Fourier transform on graphs.''   In {\it 2017
  Asia-Pacific Signal and Information Processing Association Annual Summit and
  Conference (APSIPA ASC)},   105--110.
\newblock IEEE.

\bibitem[\protect\citeauthoryear{Kutay, Ozaktas, Ankan, and Onural}{Kutay
  et~al.}{}]{kutay1997optimal}
A.~Kutay, H.~M. Ozaktas, O.~Ankan, and L.~Onural.
\newblock
\newblock ``Optimal filtering in fractional Fourier domains.''  {\it IEEE
  Transactions on Signal Processing\/}~45, no. 5 (1997): 1129--1143.

\bibitem[\protect\citeauthoryear{Ozturk, Ozaktas, Gezici, and Ko{\c{c}}}{Ozturk
  et~al.}{}]{9435933}
C.~Ozturk, H.~M. Ozaktas, S.~Gezici, and A.~Ko{\c{c}}.
\newblock
\newblock ``Optimal Fractional Fourier Filtering for Graph Signals.''  {\it
  IEEE Transactions on Signal Processing\/}69 (2021): 2902--2912.

\bibitem[\protect\citeauthoryear{Haykin}{Haykin}{}]{ipix}
S.~Haykin.
\newblock ``\textit{The {McMaster} {IPIX} {Radar} {Sea} {Clutter}
  {Database}}.''   \url{http://soma.ece.mcmaster.ca/ipix/index.html}, .

\end{thebibliography}

\end{document}